%% file: TensorRingDecomposition.tex
\documentclass[10pt,journal,compsoc,twocolumn]{IEEEtran}
\pdfoutput=1

\usepackage{pbox}
\usepackage{url}
\usepackage{multirow}
\usepackage{hyperref}
\usepackage{color}
\usepackage{textcomp}
\usepackage{multirow}
\usepackage[cmex10]{amsmath}
\usepackage{amssymb}
\usepackage{stfloats}
\usepackage{amsthm}
\usepackage{bbm}
\usepackage[nocompress]{cite}
\usepackage[pdftex]{graphicx}
\graphicspath{{./figures/}{./biography/} }
\DeclareGraphicsExtensions{.pdf,.jpeg,.png}
\usepackage[tight,normalsize,sf,SF]{subfigure}
\DeclareMathSizes{10}{10}{10}{10}
\usepackage{algpseudocode}
\usepackage{algorithm}

\input{DefinedComand}

\begin{document}

\title{Tensor Ring Decomposition}

\author{Qibin Zhao, 
        Guoxu Zhou,
        Shengli Xie,
        Liqing Zhang, 
        Andrzej Cichocki ~\IEEEmembership{Fellow,~IEEE} 
\IEEEcompsocitemizethanks{\IEEEcompsocthanksitem Q. Zhao is with Laboratory for Advanced Brain Signal Processing, RIKEN Brain Science Institute, Japan.
\IEEEcompsocthanksitem G. Zhou is with the School of Automation at Guangdong University of Technology,
Guangzhou 510006.
\IEEEcompsocthanksitem S. Xie is with the School of Automation at Guangdong University of Technology,
Guangzhou 510006.
\IEEEcompsocthanksitem L. Zhang is with MOE-Microsoft Laboratory for Intelligent Computing and Intelligent Systems and
Department of Computer Science and Engineering, Shanghai Jiao Tong University, China.
\IEEEcompsocthanksitem A. Cichocki is with Laboratory for Advanced Brain Signal Processing,
RIKEN Brain Science Institute, Japan and Systems Research Institute in Polish Academy of Science,  Warsaw,  Poland.
}
\thanks{} }

\markboth{Q. Zhao \MakeLowercase{\textit{et al.}} Tensor Ring Decomposition}%
{}

\IEEEcompsoctitleabstractindextext{%

\begin{abstract}
Tensor networks have in recent years emerged as the powerful tools for solving the large-scale optimization problems. One of the most popular tensor network is tensor train (TT) decomposition that acts as the building blocks for the complicated tensor networks. However, the TT decomposition highly depends on permutations of tensor dimensions, due to its strictly sequential multilinear products over latent cores, which leads to difficulties in finding the optimal TT representation. In this paper, we introduce a fundamental tensor decomposition model to represent a large dimensional tensor by a circular multilinear products over a sequence of low dimensional cores, which can be graphically interpreted as a cyclic interconnection of 3rd-order tensors, and thus termed as tensor ring (TR) decomposition.  The key advantage of TR model is the circular dimensional permutation invariance which is gained by employing the trace operation and treating the latent cores equivalently. TR model can be viewed as a linear combination of TT decompositions, thus obtaining the powerful and generalized representation abilities. For optimization of latent cores, we present four different algorithms based on the sequential SVDs, ALS scheme, and block-wise ALS techniques. Furthermore, the mathematical properties of TR model are investigated, which shows that the basic multilinear algebra can be performed efficiently by using TR representations and the classical tensor decompositions can be conveniently transformed into the TR representation. Finally, the experiments on both synthetic signals and real-world datasets were conducted to evaluate the performance of different algorithms.
\end{abstract}
}
\maketitle
\IEEEdisplaynotcompsoctitleabstractindextext
\IEEEpeerreviewmaketitle

\IEEEraisesectionheading{\section{Introduction}\label{sec:introduction}}
\IEEEPARstart{T}{ensor}
decompositions aim to represent a higher-order (or high-dimensional) tensor data by  multilinear operations over the latent factors, which have attracted considerable attentions in a  variety of fields including machine learning, signal processing, psychometric, chemometrics, quantum physics and brain science~\cite{Kolda09,Cichocki2015}. Specifically, the Canonical Polyadic decomposition (CPD)~\cite{bro1997parafac,DeLathauwer:2006:Link,Comon09CPD, CPD-Comon15} can approximate an observed tensor by a sum of rank-one tensors, which requires $\mathcal{O}(dnr)$ parameters. The Tucker decomposition~\cite{tucker1966some,HOOI:Lathauwer:2000,Oseledets2008} attempts to approximate tensors by a core tensor and several factor matrices, which requires $\mathcal{O}(dnr + r^d)$ parameters. These two models have been widely investigated and applied in various real-world applications, e.g.,~\cite{YokotaZCY15,caiafa2013computing, HeZ2010-2,Acar2006,Beckmann2005,Lee2009,qiinfinite, ZhouIP15, Zhou-PIEEE, Yu2012tensorDA,Miwakeichi,zhao2015bayesian,zhao2016bayesianrobust}. In general, CPD provides a compact representation but with difficulties in finding the optimal solution, while Tucker is stable and flexible but its number of parameters scales exponentially to the tensor order.

Recently, tensor networks, considered as the generalization of tensor decompositions, have emerged as the potentially powerful tools for analysis of large-scale tensor data~\cite{Orus2013,Huckle2013,Sachdev09,Murg-TTNS15}. The main concept is to transform a large-scale optimization problem into a set of small-scale tractable optimization problems, which can be achieved by representing a higher-order tensor as the interconnected lower-order tensors~\cite{CichockiSISA,hubener2010concatenated,espig2011optimization}. The most popular tensor network is \emph{tensor train / matrix product states} (TT/MPS) representation, which requires $\mathcal{O}(dnr^2)$ parameters and thus potentially allows for the efficient treatment of a higher dimensional tensor~\cite{oseledets2010tt, oseledets2011tensor, Holtz-TT-2012}. Another important tensor network is the hierarchical Tucker (HT) format~\cite{hTucker1,hTucker}, in which a recursive, hierarchical construction of Tucker type is employed for low-dimensional representation. The usefulness of these tensor formats is currently being investigated for a variety of high-dimensional problems~\cite{MRF-TT14,Dolgov2012fast}.  For example, TT-format has be successfully applied to different kinds of large-scale problems in numerical analysis, which includes the optimization of Rayleigh quotient, e.g., \emph{density matrix renormalization group} (DMRG)~\cite{jeckelmann2002dynamical,dolgovEIG2013}, the eigenvalue or singular value problem~\cite{kressner2014low, Lee-SIMAX-SVD}, and the approximate solution of linear systems, e.g., \emph{alternating minimal energy} (AMEn)~\cite{dolgov2014alternating}. One main advantage of HT and TT formats lies in the fact that the representation of higher-order tensors is reduced to $d$ tensors of order at most 3. Hence, the HT/TT formats are thus formally free from the cures of dimensionality. At the same time, new tensor formats have been proposed, e.g. the \emph{quantized tensor train} (QTT)~\cite{khoromskij2011dlog,oseledets2010approximation} and the QTT-Tucker~\cite{dolgov2013two}, which have been also applied to large-scale optimization problems~\cite{lebedeva2011tensor,dolgov2014alternating,dolgov2012superfast}.

Since TT decomposition acts as the building blocks for the complicated tensor networks, it is essentially important and necessary to investigate its properties deeply.   The principle of TT decomposition is to approximate each tensor element by a sequential products of matrices, where the first and the last matrices are vectors to ensure the scalar output~\cite{oseledets2011tensor}. It was shown in~\cite{oseledets2010tt} that TT decomposition with minimal possible compression ranks always exists and can be computed by a sequence of SVD decompositions, or by the cross approximation algorithm. In \cite{Holtz-TT-2012,rohwedder2013local}, TT decomposition is optimized by a suitable generalization of the alternating least squares (ALS) algorithm and modified ALS (MALS) algorithm which facilitates the self-adaptation of ranks either by using SVDs or by employing a greedy algorithm.  The tensor completion by optimizing the low-rank TT representations can be achieved by alternating directions fitting~\cite{Grasedyck15TTcompl} or by nonlinear conjugate gradient scheme within the framework of Riemannian optimization~\cite{steinlechner2015riemannian}. Although TT format has been widely applied in numerical analysis and mathematic field, there are only few studies addressing its applications to real dataset in machine learning field, such as image classification and completion~\cite{novikov2015tensorizing,benguamatrix,phien2016efficient}. The limitations of TT decomposition include that i) the constraint on TT-ranks, i.e., $r_1=r_{d+1}=1$, leads to the limited representation ability and flexibility; ii) TT-ranks always have a fixed pattern, i.e., smaller for the border cores and larger for the middle cores, which might not be the optimum for specific data tensor;  iii) the multilinear products of cores in TT must follow a strict order such that the optimized TT cores highly depend on the permutation of tensor dimensions.  Hence, finding the optimal permutation remains a challenging problem.

By taking into account these limitations of TT decomposition, we introduce a new type of tensor decomposition which can be considered as a generalization of the TT model. First of all, we consider to relax the condition over TT-ranks, i.e., $r_1=r_{d+1}=1$, leading to the enhanced representation ability. Secondly, the strict ordering of multilinear products between cores should be alleviated. Third, the cores should be  treated equivalently by making the model symmetric. To this end, we found these goals can be achieved by simply employing the trace operation. More specifically, we consider that each tensor element is approximated by performing a trace operation over the sequential multilinear products of cores. Since the trace operation ensures a scalar output, $r_1=r_{d+1}=1$ is not necessary. In addition, the cores can be circularly shifted and treated equivalently due to the properties of trace operation. By using the graphical illustration (see Fig.~\ref{fig:TRD}), this concept implies that the cores are interconnected circularly, which looks like a ring structure. Hence, we call this model as tensor ring (TR) decomposition and its cores as tensor ring (TR) representations. Although the similar concept has been mentioned and called MPS or tensor chain in few literatures~\cite{khoromskij2011dlog, espig2011optimization,perez2006matrix}, the algorithms and properties have not well explored yet. In this paper, the optimization algorithms for TR decomposition will be investigated, whose objective is to represent a higher-order tensor by the TR format that is potentially powerful for large-scale multilinear optimization problems.

The paper is organized as follows. In Section \ref{sec:trm}, the TR model is presented in several different forms together with its basic feature. Section \ref{sec:alg} presents four different algorithms for TR decomposition. In Section \ref{sec:property}, we demonstrate how the basic multilinear algebra can be performed by using the TR format. The relations with existing tensor decompositions are presented in Section \ref{sec:relation}.  Section \ref{sec:experiment} shows experimental results on both synthetic and real-world dataset, followed by conclusion in Secition \ref{sec:conclusion}.


\section{Tensor Ring Model}
\label{sec:trm}
The tensor ring (TR) decomposition aims to represent a high-order (or high-dimensional) tensor by a sequence of 3rd-order tensors that are multiplied circularly. Specifically, let $\tensor{T}$ be a $d$th-order  tensor of size $n_1\times n_2\times \cdots\times n_d$, denoted by $\tensor T\in\mathbb{R}^{n_1\times \cdots\times n_d}$, TR representation is to decompose it into a sequence of latent tensors $\tensor Z_k\in\mathbb{R}^{r_k\times n_k\times r_{k+1}}, k=1,2,\ldots, d$, which can be expressed in an element-wise form given by
\begin{equation}
\label{eq:TRD1}
\begin{split}
T(i_1,i_2,\ldots,i_d) =& \text{Tr}\left\{\mat Z_1(i_1)\mat Z_2(i_2)\cdots \mat Z_d(i_d)\right\} ,\\
=& \text{Tr}\left\{\prod_{k=1}^d \mat Z_k(i_k)\right\}.
\end{split}
\end{equation}
$T(i_1,i_2,\ldots,i_d)$ denotes $(i_1,i_2,\ldots,i_d)$th element of the tensor. $\mat Z_k(i_k)$ denotes the $i_k$th lateral slice matrix of the latent tensor $\tensor Z_k$, which is of size $r_k\times r_{k+1}$. Note that any two adjacent latent tensors, $\tensor Z_k$ and $\tensor Z_{k+1}$, have an equivalent dimension $r_{k+1}$ on their corresponding mode. The last latent tensor $\tensor Z_d$ is of size $r_d\times n_d\times r_1$, i.e., $r_{d+1}=r_1$, which ensures the product of these matrices is a square matrix. These prerequisites play  key roles in TR decomposition, resulting in some important numeric properties.  For simplicity,  the latent tensor $\tensor Z_k$ can be also called $k$th-\emph{core} (or \emph{node}). The size of cores, $r_k, k=1,2,\ldots, d$,  collected and denoted by a vector $\vect r = [r_1, r_2,\ldots, r_d]^T$ are called \emph{TR-ranks}. From (\ref{eq:TRD1}), we can observe that the $T(i_1,i_2,\ldots,i_d)$ is equivalent to the trace of a sequential product of matrices $\{\mat Z_k(i_k)\}$. To further describe the concept, we can also rewrite (\ref{eq:TRD1})  in the index form, which is
\begin{multline}
\label{eq:TRD2}
T(i_1,i_2,\ldots,i_d) = \sum_{\alpha_1,\ldots,\alpha_d =1}^{r_1,\ldots,r_d}  \prod_{k=1}^d Z_k(\alpha_{k},i_k,\alpha_{k+1}). 
\end{multline}
Note that $\alpha_{d+1}=\alpha_1$  due to the trace operation.  $\forall k\in\{1,\ldots,d\}, 1\leq \alpha_k \leq r_k, 1\leq i_k \leq n_k$, where $k$ is the index of  tensor modes (dimensions); $\alpha_k$ is the index of  latent dimensions; and $i_k$ is the index of data dimensions. From (\ref{eq:TRD2}), we can also easily express TR decomposition in the tensor form, given by
\begin{equation}
\tensor T = \sum_{\alpha_1,\ldots,\alpha_d=1}^{r_1,\ldots,r_d}\mat z_1(\alpha_1,\alpha_2)\circ \mat z_2(\alpha_2,\alpha_3) \circ \cdots \circ \mat z_d(\alpha_{d},\alpha_1),
\end{equation}
where the symbol `$\circ$' denotes the outer product of vectors and $\mat z_k(\alpha_k,\alpha_{k+1})\in\mathbb{R}^{n_k}$ denotes the ($\alpha_k,\alpha_{k+1}$)th mode-2 fiber of tensor $\tensor Z_k$. This indicates that the whole tensor can be decomposed into a sum of rank-1 tensors that are generated by $d$ vectors taken from each core respectively. The number of parameters in TR representation is $\mathcal{O}(dnr^2)$, which is linear to the tensor order $d$.

The TR representation can be also illustrated graphically by a linear tensor network as shown in Fig.~\ref{fig:TRD}. The node represents a tensor (including matrix and vector) whose order is denoted by the number of edges. The number beside the edges specifies the size of each mode (or dimension). The connection between two nodes denotes a multilinear product operator between two tensors on a specific mode, also called tensor contraction, which corresponds to the summation over the indices of that mode. As we can see from (\ref{eq:TRD2}), $\tensor Z_1$  and $\tensor Z_2$ is multiplied along one dimension indexed by $\alpha_2$, which is thus denoted by a connection together with the size of that mode (i.e., $r_2$) in the graph. It should be noted that $\tensor Z_d$ is connected to $\tensor Z_1$ by the summation over the index $\alpha_1$, which corresponding to the trace operation.  From the graphical representation and mathematic expression in (\ref{eq:TRD1}),  we can easily derive that  TR representation is a circular multilinear products of a sequence of 3rd-order tensors, resulting in that the sequence can be shifted circularly without changing the result essentially, which corresponds to a circular shift of tensor modes. Since our model graphically looks like a ring and its multilinear operations can be circularly shifted, we thus  call it naturally as \emph{tensor ring decomposition}. For simplicity, we denote TR decomposition by $\tensor T = \Re(\tensor Z_1, \tensor Z_2, \ldots, \tensor Z_d)$.
\begin{figure}[htbp]
  \centering
  \includegraphics[width=1\columnwidth]{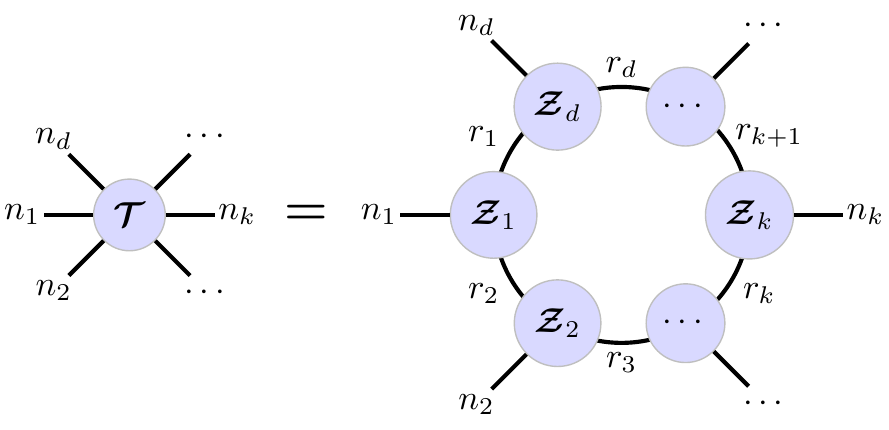}\\
  \caption{A graphical representation of the tensor ring decomposition }
  \label{fig:TRD}
\end{figure}

\begin{theorem}{\bf Circular dimensional permutation invariance}.
\label{theorem:invariance}
Let  $\tensor T\in\mathbb{R}^{n_1\times n_2\times \ldots\times n_d}$ be a $d$th-order tensor and its TR decomposition is given by  $\tensor T = \Re(\tensor Z_1, \tensor Z_2, \ldots, \tensor Z_d)$. If we define ${\overleftarrow {\tensor {T}}^k}\in \mathbb{R}^{n_{k+1}\times \cdots\times n_d\times n_1\times\cdots\times n_{k}}$ as circularly shifting the dimensions of $\tensor T$ by k, then we have ${\overleftarrow {\tensor {T}}^k} =\Re(\tensor Z_{k+1}, \ldots, \tensor Z_d, \tensor Z_{1},\ldots \tensor Z_k)$.
\end{theorem}
It is obvious that (\ref{eq:TRD1}) can be easily rewritten as
\begin{multline}
T(i_1, i_2,\ldots, i_d) = \text{Tr}(\mat Z_2(i_2), \mat Z_3(i_3),\ldots, \mat Z_d(i_d), \mat Z_1(i_1))\\ =\cdots =  \text{Tr}(\mat Z_d(i_d), \mat Z_1(i_1),\ldots, \mat Z_{d-1}(i_{d-1})).
\end{multline}
Therefore, we have
${\overleftarrow {\tensor {T}}^k} =\Re(\tensor Z_{k+1}, \ldots, \tensor Z_d, \tensor Z_{1},\ldots, \tensor Z_k)$.

It should be noted that this property is an essential feature that distinguishes TR decomposition from the TT decomposition. For TT decomposition, the product of matrices must keep a strictly sequential order, which results in that the cores for representing the same tensor with a circular dimension shifting cannot keep invariance. Hence, it is necessary to choose an optimal dimensional permutation when applying the TT decomposition.

\section{Learning Algorithms}
\label{sec:alg}
In this section, we develop several algorithms to learn the TR model. Since the exact tensor decompositions usually require heavy computation and storage, we focus on the low-rank tensor approximation under the TR framework. The selection of the optimum TR-ranks $\vect r \in \mathbb{R}^{d}$ is a challenging model selection problem. In general, $\vect r$ can be manually given, or be optimized based on the specific objective function such as nuclear norm or maximum marginal likelihood. Since the true noise distribution is unknown in practice,  we usually prefer to a low-rank approximation of the data with a relative error that can be controlled in an arbitrary scale. Therefore, given a tensor $\tensor T$, our main objective is to seek a set of cores which can approximate $\tensor T$ with a prescribed relative error $\epsilon_{p}$, while the  TR-ranks are minimum, i.e., \begin{equation}
\begin{split}
 \min_{\tensor Z_1,\ldots,\tensor Z_d}: &{\quad \vect r}\\
\text{s. t.} : &\quad \|\tensor{T}-\Re(\tensor Z_1, \tensor Z_2, \ldots, \tensor Z_d)\|_F \leq \epsilon_{p} \|\tensor T\|_F.
\end{split}
\end{equation}

\begin{definition}
Let $\tensor T \in\mathbb{R}^{n_1\times n_2\times\cdots\times n_d}$ be a $d$th-order tensor. The \emph{$k$-unfolding} of $\tensor T$ is a matrix, denoted by $\mat T_{\langle k\rangle}$ of size $\prod_{i=1}^k n_i \times \prod_{i=k+1}^d n_{i}$, whose elements are defined by
\begin{equation}
 T_{\langle k\rangle}(\overline{i_1\cdots i_k},\overline{i_{k+1}\cdots i_d}) = T(i_1,i_2,\ldots, i_d),
\end{equation}
where the first $k$ indices enumerate the rows of $\mat T_{\langle k\rangle}$, and the last $d-k$ indices for its columns.
\end{definition}

\begin{definition}
The \emph{mode-$k$ unfolding} matrix of $\tensor T $ is denoted by $\mat T_{[k]}$ of size $n_k \times \prod_{j\neq k} n_j$ with its elements defined by
\begin{equation}
T_{[k]}(i_k, \overline{i_{k+1}\cdots i_{d}i_1 \cdots i_{k-1}}) = T(i_1,i_2, \ldots, i_d),
\end{equation}
where $k$th index enumerate the rows of $\mat T_{[k]}$, and the rest $d-1$ indices for its columns. Note that the \emph{classical mode-$k$ unfolding} matrix is denoted by $\mat T_{(k)}$ of size $n_k \times \prod_{j\neq k} n_j$ and defined by
\begin{equation}
T_{(k)}(i_k, \overline{i_{1}\cdots i_{k-1}i_{k+1} \cdots i_{d}}) = T(i_1,i_2, \ldots, i_d).
\end{equation}
The difference between these two types of mode-$k$ unfolding operations lie in the ordering of indices associated to the $d-1$ modes, which corresponds to a specific dimensional permutation performed on $\tensor T$. We use these two type of definitions for clarity and notation simplicity.
\end{definition}

\begin{definition}
Let $\tensor T = \Re(\tensor Z_1, \tensor Z_2, \ldots, \tensor Z_d)$ be a TR representation of $d$th-order tensor, where $\tensor Z_k\in\mathbb{R}^{r_k\times n_k\times r_{k+1}}, k=1,\ldots,d$ be a sequence of cores.  Since the adjacent cores  $\tensor Z_k$ and $\tensor Z_{k+1}$ have an equivalent mode size $r_{k+1}$, they can be merged into a single core by multilinear products, which is defined by $\tensor Z^{(k,k+1)}\in\mathbb{R}^{r_k\times n_k n_{k+1}\times r_{k+2}}$ whose lateral slice matrices are given by
\begin{equation}
\label{eq:mergecores}
\mat Z^{(k,k+1)}(\overline{i_k i_{k+1}}) =  \mat Z_k(i_k)\mat Z_{k+1}(i_{k+1}).
\end{equation}
 Note that $\tensor Z_k, k=1,\ldots,d$ forms a circular sequence, implying that $\tensor Z_d$ is linked to $\tensor Z_1$ as well. This merging operation can be extended straightforwardly to  multiple linked cores.

The new core obtained by merging multiple linked cores $\tensor Z_1, \ldots, \tensor Z_{k-1}$, called a \emph{subchain},
is defined and denoted by $\tensor Z^{<k}\in\mathbb{R}^{r_1\times \prod_{j=1}^{k-1} n_j \times r_k}$ whose lateral slice matrices are given by
\begin{equation}
\label{eq:firstKcores}
\mat Z^{<k}(\overline{i_1 \cdots i_{k-1}}) = \prod_{j=1}^{k-1} \mat Z_j(i_j).
\end{equation}
Similarly, the subchain tensor by merging multiple linked cores $\tensor Z_{k+1},\ldots, \tensor Z_d$ is denoted by $\tensor Z^{>k}\in\mathbb{R}^{r_{k+1}\times \prod_{j=k+1}^d n_j \times r_1}$ whose lateral slice matrices are defined as
\begin{equation}
\label{eq:lastKcores}
\mat Z^{>k}(\overline{i_{k+1} \cdots i_{d}}) = \prod_{j=k+1}^{d} \mat Z_j(i_j).
\end{equation}
The subchain tensor by merging all cores except $k$th core $\tensor Z_k$, i.e., $\tensor Z_{k+1},\ldots, \tensor Z_d, \tensor Z_1, \ldots, \tensor Z_{k-1}$, is denoted by $\tensor Z^{\neq k}\in\mathbb{R}^{r_{k+1}\times \prod_{j=1,j\neq k}^{d} n_j\times r_k}$ whose slice matrices are defined by
\begin{equation}
\label{eq:subchain3}
\mat Z^{\neq k}(\overline{i_{k+1} \cdots i_{d}i_1\ldots i_{k-1} }) = \prod_{j=k+1}^{d} \mat Z_j(i_j)\prod_{j=1}^{k-1} \mat Z_j(i_j).
\end{equation}
\end{definition}
Analogously, we can also define subchains of $\tensor Z^{\leq k}$, $\tensor Z^{\geq k}$ and $\tensor Z^{\neq (k, k+1)}$ in the same way. Note that a special subchain by merging all cores is denoted by $\tensor Z^{(1:d)}$ of size $r_1 \times \prod_{j=1}^d n_j\times r_1$ whose mode-2 fibers $\mat Z^{(1:d)}(\alpha_1, :, \alpha_1), \alpha_1=1,\ldots, r_1$ can be represented as TT representations, respectively.

\subsection{Sequential SVDs algorithm}
We propose the first algorithm for computing the TR decomposition using $d$ sequential SVDs. This algorithm will be called \emph{TR-SVD} algorithm.

\begin{theorem}
Let us assume $\tensor T$ can be represented by a TR decomposition. If  the $k$-unfolding matrix $\mat T_{\langle k\rangle}$ has $Rank(\mat T_{\langle k \rangle})=R_{k+1}$, then there exists a TR decomposition with TR-ranks $\vect r$ which satisfies that $ \exists k, r_1r_{k+1} \leq R_{k+1}$.
\begin{proof}
We can express TR decomposition in the form of $k$-unfolding matrix,
\begin{equation}
\label{eq:TRSVD1}
\begin{split}
& T_{\langle k\rangle}(\overline{i_1\cdots i_k}, \overline{i_{k+1}\cdots i_d})  = \text{Tr}\left\{\mat Z_1(i_1)\mat Z_2(i_2)\cdots \mat Z_d(i_d)\right\}\\
&\qquad = \text{Tr}\left\{ \prod_{j=1}^k \mat Z_j(i_j)  \prod_{j=k+1}^d \mat Z_j(i_j) \right\}\\
&\qquad = \left\langle \text{vec}\left(\prod_{j=1}^k \mat Z_j(i_j)\right), \text{vec}\left(\prod_{j=d}^{k+1} \mat Z^T_j(i_j) \right)   \right\rangle.
\end{split}
\end{equation}
According to the definitions in (\ref{eq:firstKcores})(\ref{eq:lastKcores}),  (\ref{eq:TRSVD1}) can be also rewritten as
\begin{equation}
\label{eq:TRSVD2}
\begin{split}
& T_{\langle k\rangle}(\overline{i_1\cdots i_k}, \overline{i_{k+1}\cdots i_d}) \\
& = \sum_{\alpha_1\alpha_{k+1}} Z^{\leq k}\left(\overline{i_1\cdots i_k}, \overline{\alpha_1\alpha_{k+1}}\right)  Z^{>k}\left(\overline{\alpha_1\alpha_{k+1}},\overline{i_{k+1}\cdots i_d}\right).
\end{split}
\end{equation}
Hence, we can obtain that $\mat T_{\langle k\rangle} = \mat Z^{\leq k}_{(2)} (\mat Z^{>k}_{[2]})^T$, where the subchain $\mat Z^{\leq k}_{(2)}$ is of size $\prod_{j=1}^k n_j\times r_1r_{k+1}$, and $\mat Z^{>k}_{[2]}$ is of size $\prod_{j=k+1}^d n_{j}\times r_1r_{k+1}$. Since the rank of $\mat T_{\langle k\rangle}$ is $R_{k+1}$,  we can obtain that $r_1r_{k+1} \leq R_{k+1}$.
\end{proof}
\end{theorem}

For TT-SVD algorithm, we usually need to choose a specific mode as the start point (e.g., the first mode).
According to (\ref{eq:TRSVD1})(\ref{eq:TRSVD2}), TR decomposition can be easily written as
\begin{equation}
T_{\langle 1\rangle}(i_1, \overline{i_2\cdots i_d}) = \sum_{\alpha_1,\alpha_2} Z^{\leq 1}(i_1, \overline{\alpha_1\alpha_2})Z^{>1}(\overline{\alpha_1\alpha_2}, \overline{i_2\cdots i_d}).
\end{equation}
Since the low-rank approximation of $\mat T_{\langle 1\rangle}$ can be easily obtained by the truncated SVD, which is
\begin{equation}
\mat T_{\langle 1 \rangle} = \mat U \Sigma\mat V^T +\mat E_{1},
\end{equation}
the first core $\tensor Z_1 (i.e., \tensor Z^{\leq 1}) $ of size $r_1\times n_1\times r_2$ can be obtained by the proper reshaping and permutation of $\mat U$ and the subchain $\tensor Z^{>1}$ of size $r_2\times \prod_{j=2}^d n_j\times r_1$ is obtained by the proper reshaping and permutation of $\mat \Sigma \mat V^T$, which corresponds to the rest $d-1$ dimensions of $\tensor T$.  Subsequently, we can further reshape the subchain $\tensor Z^{>1}$ as a matrix $\mat Z^{>1}\in\mathbb{R}^{r_2n_2\times \prod_{j=3}^d n_j r_1}$ which thus can be written as
\begin{equation}
Z^{>1}(\overline{\alpha_2 i_2}, \overline{i_3\cdots i_d\alpha_1}) = \sum_{\alpha_3}Z_2(\overline{\alpha_2 i_2}, \alpha_3) Z^{>2}(\alpha_3, \overline{i_3\cdots i_d\alpha_1}).
\end{equation}
By applying truncated SVD, i.e., $\mat Z^{>1} = \mat U\Sigma\mat V^T + \mat E_2$, we can obtain the second core $\tensor Z_2$ of size $(r_2\times n_2\times r_3)$ by appropriately reshaping $\mat U$ and the subchain $\tensor Z^{>2}$ by proper reshaping of $\mat \Sigma \mat V^T$. This procedure can be performed sequentially to obtain all $d$ cores $\tensor Z_k, k=1,\ldots, d$.

As proved in \cite{oseledets2011tensor}, the approximation error by using such sequential SVDs is given by
\begin{equation}
\|\tensor{T}-\Re(\tensor Z_1, \tensor Z_2, \ldots, \tensor Z_d)\|_F \leq \sqrt{\sum_{k=1}^{d-1} \|\mat E_k\|_F^2}.
\end{equation}
Hence, given a prescribed relative error $\epsilon_p$, the truncation threshold $\delta$ can be set to $\frac{\epsilon_p}{\sqrt{d-1}}\|\tensor T\|_F$. However, considering that $\|\mat E_1\|_F$ corresponds to two ranks including both $r_1$ and $r_2$, while $\mat \|\mat E_k\|_F, \forall k>1$ correspond to only one rank $r_{k+1}$. Therefore, we modify the truncation threshold as
\begin{equation}
\label{eq:TRSVDthreshold}
\delta_k =\left \{\begin{array}{c}
             \sqrt{2}\epsilon_p \|\tensor T\|_F / \sqrt{d}, \quad k=1, \\
             \epsilon_p\|\tensor T\|_F / \sqrt{d}, \quad k>1. \\
           \end{array} \right.
\end{equation}
Finally, the TR-SVD algorithm is summarized in Alg.~\ref{alg:TR-SVD}.

\renewcommand{\algorithmicrequire}{\textbf{Input:}}
\renewcommand{\algorithmicensure}{\textbf{Output:}}
\begin{algorithm}
\caption{TR-SVD}
\label{alg:TR-SVD}
\begin{algorithmic}[1]
\Require A $d$th-order tensor $\tensor T$ of size $(n_1\times\cdots\times n_d)$ and the prescribed relative error $\epsilon_p$.
\Ensure  Cores $\tensor Z_k, k=1,\ldots,d$ of TR decomposition and the TR-ranks $\mat r$.
\State Compute truncation threshold $\delta_k$ for $k=1$ and $k>1$.
\State Choose one mode as the start point (e.g., the first mode) and obtain the $1$-unfolding matrix $\mat T_{\langle 1\rangle}$.
\State Low-rank approximation by applying $\delta_1$-truncated SVD: $\mat T_{\langle 1 \rangle} = \mat U\Sigma\mat V^T + \mat E_1$.
\State Split ranks $r_1,r_2$ by
\begin{equation}
\nonumber
\begin{split}
\min_{r_1,r_2} \quad \|r_1-r_2\|, \quad
 s. t. \quad r_1r_2 = \text{rank}_{\delta_1}(\mat T_{\langle 1\rangle}).
\end{split}
\end{equation}
\State $\tensor Z_1 \gets \text{permute}(\text{reshape}(\mat U, [ n_1, r_1, r_2]), [2,1,3])$.
\State $\tensor Z^{>1} \gets \text{permute}(\text{reshape}(\mat \Sigma\mat V^T, [r_1,r_2,\prod_{j=2}^d n_j]), [2,3,1])$.
\For{ $k=2$ to $d-1$}
\State $\mat Z^{>k-1} = \text{reshape}(\tensor Z^{>k-1}, [r_k n_k, n_{k+1}\cdots n_d r_1])$.
\State Compute $\delta_k$-truncated SVD:
             \begin{equation}\nonumber\mat Z^{>k-1} = \mat U\Sigma\mat V^T + \mat E_k.\end{equation}
\State $r_{k+1} \gets \text{rank}_{\delta_k}(\mat Z^{>k-1})$.
\State $\tensor Z_k \gets \text{reshape}(\mat U, [r_k, n_k, r_{k+1}])$.
\State $\tensor Z^{>k}\gets \text{reshape}(\Sigma\mat V^T, [r_{k+1}, \prod_{j=k+1}^{d} n_j, r_1])$.
\EndFor
\end{algorithmic}
\end{algorithm}

The cores obtained by TR-SVD algorithm are left-orthogonal, which is $\mat Z^T_{k\langle 2\rangle}\mat Z_{k\langle 2\rangle}=\mat I$, for $k=2,\ldots, d-1$. It should be noted that TR-SVD is a non-recursive algorithm that does not need iterations for convergence. However, it might obtain different representations by choosing a different mode as the start point. This indicates that TR-ranks $\mat r$ is not necessary to be the global optimum in TR-SVD. Therefore, we consider to develop other algorithms that can find the optimum TR-ranks and are independent with the start point.

\subsection{ALS algorithm}
In this section, we introduce an algorithm for TR decomposition by employing alternating least squares (ALS). The ALS algorithm has been widely applied to most tensor decomposition models such as CP and Tucker decompositions~\cite{Kolda09,Holtz-TT-2012}. The main concept of ALS is optimizing one core while the other cores are fixed, and this procedure will be repeated until some convergence criterion is satisfied. Given a $d$th-order tensor $\tensor T$, our goal is to optimize the cores with a given TR-ranks $\vect r$, i.e.,
\begin{equation}
\label{eq:alsobj}
\min_{\tensor Z_1,\ldots, \tensor Z_d }\|\tensor T - \Re(\tensor Z_1, \ldots, \tensor Z_d)\|_F.
\end{equation}

\begin{theorem}
\label{Theorem:kunfolding}
Given a TR decomposition $\tensor T = \Re(\tensor Z_1, \ldots, Z_d)$, its mode-$k$ unfolding matrix can be written as
\begin{equation}\label{eq:theoremals}\mat T_{[k]} = {\mat Z_k}_{(2)} \left(\mat Z^{\neq k}_{[2]}\right)^T ,\end{equation}
where $\tensor Z^{\neq k}$ is a subchain obtained by merging $d-1$ cores, which is defined in (\ref{eq:subchain3}).
\end{theorem}
\begin{proof}
According to the TR definition in (\ref{eq:TRD2}), we have
\begin{equation}
\resizebox{1\columnwidth}{!}{$
\begin{split}
&T(i_1,i_2,\ldots,i_d) \\
& = \sum_{\alpha_1,\ldots,\alpha_d}Z_1(\alpha_1,i_1,\alpha_2)Z_2(\alpha_2,i_2,\alpha_3) \cdots Z_d(\alpha_{d},i_d,\alpha_1)\\
& = \sum_{\alpha_{k},\alpha_{k+1}} \Big\{ Z_k(\alpha_k,i_k, \alpha_{k+1}) \sum_{\substack{\alpha_1,\ldots, \alpha_{k-1} \\\alpha_{k+2}\ldots,\alpha_d}} Z_{k+1}(\alpha_{k+1},i_{k+1},\alpha_{k+2})    \\
& \qquad  \cdots Z_{d}(\alpha_d,i_d,\alpha_1) Z_1(\alpha_1,i_1,\alpha_2)  \cdots Z_{k-1}(\alpha_{k-1},i_{k-1},\alpha_{k})\Big\} \\
& =\sum_{\alpha_{k},\alpha_{k+1}} \Big\{Z_k(\alpha_k,i_k, \alpha_{k+1}) Z^{\neq k} (\alpha_{k+1}, \overline{i_{k+1}\cdots i_d i_1 \cdots i_{k-1}},\\
& \qquad \alpha_{k})\Big\}.
\end{split}
$}
\end{equation}
Hence, the mode-$k$ unfolding matrix of $\tensor T$ can be expressed by
\begin{equation}
\begin{split}
& T_{[k]}(i_k, \overline{i_{k+1} \cdots i_d i_1 \cdots i_{k-1}}) = \sum_{\alpha_k\alpha_{k+1}}\Big\{ Z_k(i_k, \overline{\alpha_k \alpha_{k+1}}) \\
&\qquad Z^{\neq k}(\overline{\alpha_k\alpha_{k+1}},\overline{i_{k+1}\cdots i_d i_1\cdots i_{k-1}})\Big\}.
\end{split}
\end{equation}
This indicates a product of two matrices. By applying different mode-$k$ unfolding operations, we can easily justify the formula in (\ref{eq:theoremals}).
\end{proof}

Based on Theorem \ref{Theorem:kunfolding}, the objective function in (\ref{eq:alsobj}) can be optimized by solving $d$ subproblems alternatively. More specifically,  having fixed all but one core, the problem reduces to a linear least squares problem, which is
\begin{equation}
\min_{\mat Z_{k(2)}} \left\| \mat T_{[k]} - \mat Z_{k(2)} \left(\mat Z^{\neq k}_{[2]}\right)^T \right\|_F, \quad k=1,\ldots,d.
\end{equation}
Therefore, TR composition can be performed by ALS optimizations, which is called TR-ALS algorithm. The detailed procedure is shown in Alg.~\ref{alg:TR-ALS}.

\begin{algorithm}
\caption{TR-ALS}
\label{alg:TR-ALS}
\begin{algorithmic}[1]
\Require A $d$th-order tensor $\tensor T$ of size $(n_1\times\cdots\times n_d)$ and the predefined TR-ranks $\mat r$.
\Ensure  Cores $\tensor Z_k, k=1,\ldots,d$ of TR decomposition.
\State Initialize $\tensor Z_k\in\mathbb{R}^{r_{k}\times n_k\times r_{k+1}}$ for $k=1,\ldots, d$ as random tensors from Gaussian distribution.
\Repeat
\For{ $k=1$ to $d$}
\State Compute the subchain $\tensor Z^{\neq k}$ by using (\ref{eq:subchain3}).
\State Obtain $\mat Z^{\neq k}_{[2]}$ of size $\prod_{j=1}^d n_j / n_k \times \, r_k r_{k+1}$.
\State $\mat Z_{k(2)} \gets \arg\min \|{\mat T_{[k]} - \mat Z_{k(2)} (\mat Z_{[2]}^{\neq k}})^T\|_F$.
\State Normalize columns of $\mat Z_{k(2)}$, if $k\neq d$.
\State $\tensor Z_{k} \gets \text{permute}( \text{reshape}(\mat Z_{k(2)}, [ n_k,r_k,r_{k+1}]), [2,1,3])$.
\State Relative error $\epsilon \gets \|\tensor T - \Re(\tensor Z_1,\ldots, \tensor Z_d) \|/\|\tensor T\|_F$
\EndFor
\Until{Relative changes of $\epsilon$ is smaller than a specific threshold (e.g. $10^{-6})$, or maximum number of iterations is reached. }
\end{algorithmic}
\end{algorithm}

The cores can be initialized randomly with specified TR-ranks. The iterations repeat until some combination of stopping conditions is satisfied. Possible stopping conditions include the following: little or no improvement in the objective function; the value of objective function being smaller than a specific threshold; a predefined maximum number of iterations is reached. The normalization is performed on all cores except the last one that absorbs the weights. It should be noted that the cores are not necessary to be orthogonal in TR-ALS.

\subsection{ALS with adaptive ranks}
One important limitation of TR-ALS algorithm is that TR-ranks must be specified and fixed, which may make the algorithm difficult to obtain a desired accuracy. Although we can try different TR-ranks and select the best one, the computation cost will dramatically increase due to the large number of possibilities when high dimensional tensors are considered. Therefore, we attemp to develop an ALS algorithm for TR decomposition with adaptive ranks, which is simply called ALSAR algorithm.

The ALSAR algorithm  is initialized with equivalent TR-ranks, i.e., $r_1=r_2=\cdots = r_d=1$. The core tensors $\tensor Z_k, k=1,\ldots,d $ are initialized by random tensors which are of size $1\times n_k\times 1$. For optimization of each core tensor $\tensor Z_k$, it was firstly updated according to ALS scheme, yielding the updated approximation error $\epsilon_{old}$. Then, we attempt to increase the rank by $r_{k+1} \leftarrow r_{k+1}+1$, which implies that $\tensor Z_k, \tensor Z_{k+1}$ must be updated with the increased sizes. More specifically, based on the modified $\tensor Z_{k+1}$ by adding more random entries,  $\tensor Z_k$ is updated again yielding the new approximation error $\epsilon_{new}$. If the improvement of approximation error by increasing the rank $r_{k+1}$ satisfies a specific criteria, then the increased rank is accepted otherwise it is rejected.  The acceptance criteria can be simply expressed by
\begin{equation}
|{\epsilon_{old}-\epsilon_{new}}| > \tau {|\epsilon_{old}-\epsilon_{p}|},
\end{equation}
where $\epsilon_{p}$ denotes the desired approximation error, and $\tau$ denotes the threshold for accepting the increased rank. The reasonable choices for $\tau$ vary between $10^{-2}/d$ and $10^{-3}/d$. This procedure will be repeated untill the desired approximation error is reached. The detailed algorithm is summarized in Alg. \ref{alg:TR-ALSAR}.

This algorithm is intuitive and heuristic. During the optimization procedure, the corresponding rank is tentatively increased  followed by a decision making based on the acceptance criterion. Hence, it can achieve an arbitrary approximation error $\epsilon_p$ by rank adaptation. However, since the step of rank increasing is only one at each iteration, it might need many iterations to achieve the desired accuracy if TR-ranks are relatively large.

\begin{algorithm}
\caption{TR-ALSAR}
\label{alg:TR-ALSAR}
\begin{algorithmic}[1]
\Require A $d$-dimensional tensor $\tensor T$ of size $(n_1\times\cdots\times n_d)$ and the prescribed relative error $\epsilon_{p}$.
\Ensure  Cores $\tensor Z_k$ and TR-ranks $r_k, k=1,\ldots, d$.
\State Initialize $r_k =1$ for $k=1,\ldots, d$.
\State Initialize $\tensor Z_k\in\mathbb{R}^{r_{k}\times n_k\times r_{k+1}}$ for $k=1,\ldots, d$.
\Repeat
\For{ $k=1$ to $d$}
\State $\mat Z_{k(2)} \gets \arg\min \|{\mat T_{[k]} - \mat Z_{k(2)} (\mat Z_{[2]}^{\neq k}})^T\|_F$.
\State Evaluate relative error $\epsilon_{old}$.
\State $r_{k+1} \leftarrow r_{k+1}+1$.
\State Increase the size of $\tensor Z_{k+1}$ by random samples.
\State Repeat the steps 5-7 and evaluate the error $\epsilon_{new}$.
\State \parbox[t]{\dimexpr\linewidth-\algorithmicindent-0.1in}{Determine $\mat Z_{k(2)}$ and $r_{k+1}$ according to the specific    acceptance criterion. \strut}
\State Normalize columns of $\mat Z_{k(2)}$ , if $k\neq d$.
\State $\tensor Z_k \gets \text{permute}( \text{reshape}(\mat Z_{k(2)}, [ n_k,r_k,r_{k+1}]), [2,1,3])$.
\EndFor
\Until{The desired approximation accuracy is achieved, i.e., $\epsilon \leq \epsilon_p$.}
\end{algorithmic}
\end{algorithm}

\subsection{Block-wise ALS algorithm}
In this section, we propose a computationally efficient block-wise ALS (BALS) algorithm by utilizing truncated SVD, which facilitates the self-adaptation of ranks. The main idea is to perform the blockwise optimization followed by the separation of block into individual cores. To achieve this, we consider to merge two linked  cores, e.g., $\tensor Z_k, \tensor Z_{k+1}$, into a block (or subchain) $\tensor Z^{(k,k+1)}\in \mathbb{R}^{r_k\times n_k n_{k+1}\times r_{k+2}}$ (see definition in (\ref{eq:mergecores})). Thus, the subchain $\tensor Z^{(k,k+1)}$ can be optimized while leaving all cores except $\tensor Z_k, \tensor Z_{k+1}$ fixed. Subsequently, the subchain $\tensor Z^{(k,k+1)}$ can be reshaped into $\tilde{\mat Z}^{(k,k+1)}\in\mathbb{R}^{r_kn_k\times n_{k+1}r_{k+2}}$ and separated into a left-orthonormal core $\tensor Z_k$ and $\tensor Z_{k+1}$ by a truncated SVD,
\begin{equation}
\label{eq:BALS1}
\tilde{\mat Z}^{(k,k+1)} = \mat U \mat \Sigma \mat V^T = \mat Z_{k\langle 2\rangle }\mat Z_{k+1\langle 1\rangle},
\end{equation}
where $\mat Z_{k\langle 2\rangle }\in\mathbb{R}^{r_kn_k\times r_{k+1}} $ is the $2$-unfolding matrix of core $\tensor Z_k$, which can be set to $\mat U$, while $\mat Z_{k+1\langle 1\rangle}\in \mathbb{R}^{r_{k+1}\times n_{k+1}r_{k+2}}$ is the $1$-unfolding matrix of core $\tensor Z_{k+1}$, which can be set to $\mat \Sigma\mat V^T $. This procedure thus move on to optimize the next block cores $\tensor Z^{(k+1,k+2)}, \ldots, \tensor Z^{(d-1,d)},\tensor Z^{(d,1)}$ successively in the similar way. Note that since TR model is circular, the $d$th core  can  be also merged with the first core yielding the block core $\tensor Z^{(d,1)}$.

The key advantage of BALS algorithm is the rank adaptation which can be achieved simply by separating the block core into two cores via truncated SVD, as shown in (\ref{eq:BALS1}). The truncated rank $r_{k+1}$ can be chosen such that the approximation error is below a certain threshold. One possible choice is to use the same threshold as in TR-SVD algorithm, i.e., $\delta_k$ described in (\ref{eq:TRSVDthreshold}). However, the empirical experience show that this threshold  often leads to overfitting and the truncated rank is higher than the optimum rank. This is because that the updated block $\tensor Z^{(k,k+1)}$ during ALS iterations is not a closed form solution and many iterations are necessary for convergence. To relieve this problem, we choose the truncation threshold based on both the current and the desired approximation errors, which is
\begin{equation}
\delta = \max \left\{\epsilon\|\tensor{T}\|_F/\sqrt{d},\, \epsilon_{p}\|\tensor{T}\|_F/ \sqrt{d}\right\}.
\end{equation}
The details of BALS algorithm are described in Alg. \ref{alg:TR-BALS}.

\begin{algorithm}[htbp]
\caption{TR-BALS}
\label{alg:TR-BALS}
\begin{algorithmic}[1]
\Require A $d$-dimensional tensor $\tensor T$ of size $(n_1\times\cdots\times n_d)$ and the prescribed relative error $\epsilon_{p}$.
\Ensure  Cores $\tensor Z_k$ and TR-ranks $r_k$, $k=1,\ldots,d$.
\State Initialize $r_k =1$ for $k=1,\ldots, d$.
\State Initialize $\tensor Z_k\in\mathbb{R}^{r_{k}\times n_k\times r_{k+1}}$ for $k=1,\ldots, d$.
\Repeat $\quad k\in\text{circular}\{1,2,\ldots, d\}$;
\State Compute the subchain $\tensor Z^{\neq (k,k+1)}$ by using (\ref{eq:subchain3}).
\State \parbox[t]{\dimexpr\linewidth-\algorithmicindent-0.1in} {Obtain the mode-2 unfolding matrix $\mat Z^{\neq (k,k+1)}_{[2]}$ of size $\prod_{j=1}^d n_j/(n_k n_{k+1})\times r_k r_{k+2}$. \strut}
\State $\mat Z^{(k,k+1)}_{(2)} \gets \arg\min \left\|{\mat T_{[k]} - \mat Z^{(k,k+1)}_{(2)} \left(\mat Z_{[2]}^{\neq (k, k+1)}\right)^T}\right\|_F$.
\State Tensorization of mode-2 unfolding matrix
\begin{equation}\nonumber \tensor Z^{(k,k+1)} \gets \text{folding}(\mat Z^{(k,k+1)}_{(2)}).\end{equation}
\State Reshape the block core by
\begin{equation}\nonumber \tilde{\mat Z}^{(k,k+1)} \gets \text{reshape}(\tensor Z^{(k,k+1)}, [r_k n_k\times n_{k+1}r_{k+2} ]).   \end{equation}
\State \parbox[t]{\dimexpr\linewidth-\algorithmicindent-0.1in} {Low-rank approximation by $\delta$-truncated SVD
         $\tilde{\mat Z}^{(k,k+1)} = \mat U \mat \Sigma \mat V^T$. \strut}
\State $\tensor Z_k \gets \text{reshape}(\mat U, [r_k,n_k,r_{k+1}])$.
\State $\tensor Z_{k+1} \gets \text{reshape}(\Sigma \mat V^T, [r_{k+1},n_{k+1},r_{k+2}])$.
\State $r_{k+1} \gets \text{rank}_{\delta} (\tilde{\mat Z}^{(k,k+1)})$.
\State $k\gets k+1$.
\Until{The desired approximation accuracy is achieved, i.e., $\epsilon \leq \epsilon_p$. }
\end{algorithmic}
\end{algorithm}

\subsection{Discussions on the proposed algorithms}
We briefly summarize and discuss the  proposed  algorithms as follows. TR-SVD algorithm is generally efficient in computation due to its non-recursion property and it can approximate an arbitrary tensor as close as possible. However, its obtained TR-ranks $[r_1, \ldots, r_d]$ always follow a specific pattern, i.e., smaller ranks on the both sides and larger ranks in the middle, which might not suitable to the observed data.  The other three algorithms are based on ALS framework, resulting in that the optimum TR-ranks only depend on the observed data while the recursive procedure leads to the slow convergence.   TR-ALS is stable but requires TR-ranks to be fixed and manually given.  TR-ALSAL is able to adapt TR-ranks during optimization, but requires many iterations. TR-BALS enables us to find the optimum TR-ranks efficiently without dramatically increasing the computational cost.

\section{Properties of TR representation}
\label{sec:property}
In this section, we discuss some interesting properties of TR representation. By assuming that tensor data have been already represented as TR decompositions, i.e., a sequence of third-order cores, we justify and demonstrate that the basic operations on tensors, such as \emph{addition}, \emph{multilinear product}, \emph{Hadamard product}, \emph{inner product} and \emph{Frobenius norm}, can be performed efficiently by the appropriate operations on each individual cores. These properties are crucial and essentially important for processing large-scale or large-dimensional tensors, due to the ability of converting a large problem w.r.t. the original tensor  into many small problems w.r.t. individual cores.

\begin{theorem}
Let $\tensor T_1$ and $\tensor T_2$ be $d$th-order tensors of size $n_1\times \cdots\times n_d$. If the TR decompositions of these two tensors are $\tensor T_1 = \Re(\tensor Z_1,\ldots,\tensor Z_d)$ where $\tensor Z_k\in\mathbb{R}^{r_k\times n_k\times r_{k+1}}$ and $\tensor T_2 = \Re(\tensor Y_1,\ldots,\tensor Y_d)$, where $\tensor Y_k\in\mathbb{R}^{s_k\times n_k\times s_{k+1}}$, then the addition of these two tensors, $\tensor T_3= \tensor T_1 + \tensor T_2$, can be also represented in the TR format given by $\tensor T_3 = \Re(\tensor X_1, \ldots, \tensor X_d)$, where $\tensor X_k\in\mathbb{R}^{q_k\times n_k\times q_{k+1}}$ and $q_k = r_k+s_k$. Each core $\tensor X_k$ can be computed by
\begin{equation}
\label{eq:TRsum1}
\mat X_k(i_k) = \left(
                  \begin{array}{cc}
                    \mat Z_k(i_k) & 0 \\
                    0 & \mat Y_k(i_k) \\
                  \end{array}
                \right),
                  \begin{array}{c}
                     i_k=1,\ldots,n_k,\\
                    k=1,\ldots, d.\\
                  \end{array}
\end{equation}
\end{theorem}
\begin{proof}
According to the definition of  TR decomposition, and the cores shown in (\ref{eq:TRsum1}), the $(i_1,\ldots,i_d)$th element of tensor $\tensor T_3$ can be written as
\begin{equation}
\begin{split}
   T_3(i_1,\ldots,i_d)       =&\text{Tr}(\mat X_1(i_1)\ldots \mat X_d(i_d)) \\
               = &\text{Tr}\left(
                  \begin{array}{cc}
                     \prod_{k=1}^d \mat Z_k(i_k) & 0 \\
                    0 &  \prod_{k=1}^d \mat Y_k(i_k) \\
                  \end{array}
                \right)\\
                =& \text{Tr}\left(\prod_{k=1}^d \mat Z_k(i_k)\right) + \text{Tr}\left(\prod_{k=1}^d \mat Y_k(i_k)\right)\\
                =& T_1(i_1,\ldots,i_d) + T_2(i_1,\ldots,i_d).
\end{split}
\end{equation}
Hence, the \emph{addition} of tensors in the TR format can be performed by  merging of their cores.
\end{proof}
Note that the sizes of new cores are increased and not optimal in general. This problem can be solved by the rounding procedure \cite{oseledets2011tensor}.

\begin{theorem}
\label{theorem:TRtimesvectors}
Let $\tensor T\in\mathbb{R}^{n_1\times\cdots\times n_d}$ be a $d$th-order tensor whose TR representation is $\tensor T = \Re(\tensor Z_1,\ldots,\tensor Z_d)$ and $\mat u_k\in\mathbb{R}^{n_k}, k=1,\ldots,d$ be a set of vectors, then the multilinear products, denoted by $c=\tensor T \times_1 \mat u_1^T\times_2\cdots \times_d \mat u_d^T$, can be computed by multilinear product on each cores, which is
\begin{equation}
\label{eq:TRmultiproduct}
\begin{split}
c=\Re(\mat X_1, \ldots, \mat X_d) \; \text{where} \; \mat X_k = \sum_{i_k=1}^{n_k}\mat Z_k(i_k) u_k(i_k).
\end{split}
\end{equation}
\end{theorem}
\begin{proof}
The \emph{multilinear product} between a tensor and vectors can be expressed by
\begin{equation}
\label{eq:TRmultiproductproof}
\begin{split}
c=&\tensor T \times_1 \mat u_1^T\times_2\cdots \times_d \mat u_d^T\\
=& \sum_{i_1,\ldots,i_d} T(i_1,\ldots, i_d) u_1(i_1)\cdots u_d(i_d)\\
=& \sum_{i_1,\ldots,i_d} \text{Tr}\left(\prod_{k=1}^d \mat Z_k(i_k)\right) u_1(i_1)\cdots u_d(i_d)\\
=& \text{Tr}\left( \prod_{k=1}^d \left(\sum_{i_k=1}^{n_k} \mat Z_k(i_k)u_k(i_k)  \right)       \right).
\end{split}
\end{equation}
Thus, it can be written as a TR decomposition shown in (\ref{eq:TRmultiproduct}) where each core $\mat X_k \in\mathbb{R}^{r_k\times r_{k+1}}$ becomes a matrix. The computational complexity is equal to $\mathcal{O}(dnr^2)$.
\end{proof}
From (\ref{eq:TRmultiproductproof}), we can see that the multilinear product between $\tensor T$ and $\mat u_k, k=1,\ldots,d$ can be also expressed as an inner product of $\tensor T$ and the rank-1 tensor, i.e.,
\begin{equation}
\tensor T \times_1 \mat u_1^T\times_2\cdots \times_d \mat u_d^T = \langle \tensor T, \mat u_1 \circ \cdots\circ \mat u_d \rangle.
\end{equation}
It should be noted that the computational complexity in the original tensor form is $\mathcal{O}(dn^d)$, while it reduces to $\mathcal{O}(dnr^2 + dr^3)$ that is linear to tensor order $d$ by using TR representation.

\begin{theorem}
Let $\tensor T_1$ and $\tensor T_2$ be $d$th-order tensors of size $n_1\times \cdots\times n_d$. If the TR decompositions of these two tensors are $\tensor T_1 = \Re(\tensor Z_1,\ldots,\tensor Z_d)$ where $\tensor Z_k\in\mathbb{R}^{r_k\times n_k\times r_{k+1}}$ and $\tensor T_2 = \Re(\tensor Y_1,\ldots,\tensor Y_d)$, where $\tensor Y_k\in\mathbb{R}^{s_k\times n_k\times s_{k+1}}$, then the Hadamard product of these two tensors, $\tensor T_3= \tensor T_1 \circledast \tensor T_2$, can be also represented in the TR format given by $\tensor T_3 = \Re(\tensor X_1, \ldots, \tensor X_d)$, where $\tensor X_k\in\mathbb{R}^{q_k\times n_k\times q_{k+1}}$ and $q_k = r_k*s_k$. Each core $\tensor X_k$ can be computed by
\begin{equation}
\label{eq:TRhadama}
\mat X_k(i_k) = \mat Z_k(i_k) \otimes \mat Y_k(i_k), \quad k=1,\ldots,d.
\end{equation}
\end{theorem}
\begin{proof}
Each element in tensor $\tensor T_3$ can be written as
\begin{equation}
\begin{split}
T_3(i_1,\ldots,i_d) =& T_1(i_1,\ldots,i_d) T_2(i_1,\ldots,i_d)\\
=&\text{Tr}\left(\prod_{k=1}^d \mat Z_k(i_k)\right)\text{Tr}\left(\prod_{k=1}^d \mat Y_k(i_k)\right)\\
=& \text{Tr}\left\{\left(\prod_{k=1}^d \mat Z_k(i_k)\right)\otimes \left(\prod_{k=1}^d \mat Y_k(i_k)\right)\right\}\\
=& \text{Tr} \left\{\prod_{k=1}^d \Big( \mat Z_k(i_k) \otimes \mat Y_k(i_k) \Big) \right\}.
\end{split}
\end{equation}
Hence, $\tensor T_3$ can be also represented as TR format with its cores computed by (\ref{eq:TRhadama}), which costs $\mathcal{O}(dnq^2)$.
\end{proof}

Furthermore, one can compute the \emph{inner product} of two tensors in TR representations. For two tensors $\tensor T_1$ and $\tensor T_2$, it is defined as
\begin{equation}
\begin{split}
\langle \tensor T_1, \tensor T_2\rangle  =  \sum_{i_1,\ldots,i_d} T_3(i_1,\ldots,i_d),
\end{split}
\end{equation}
where $\tensor T_3 =  \tensor T_1 \circledast \tensor T_2$. Thus, the inner product can be computed by applying Hadamard product and then computing the multilinear product between $\tensor T_3$ and vectors of all ones, i.e., $\mat u_k=\mat 1, k=1,\ldots,d$.  In contrast to $\mathcal{O}(n^d)$ in the original tensor form, the computational complexity is equal to $\mathcal{O}(dnq^2+ dq^3)$ that is linear to $d$ by using TR representation. Similarly, we can also compute \emph{Frobenius norm} $\|\tensor T \|_F = \sqrt{\langle\tensor T, \tensor T \rangle}$ in the TR representation.

In summary, by using TR representations, many important multilinear operations can be performed by operations on their cores with smaller sizes,  resulting in that the computational complexity scales linearly to the tensor order.

\section{Relation to Other Models}
\label{sec:relation}
In this section, we discuss the relations between TR model and the classical tensor decompositions including CPD, Tucker and TT models.  All these tensor decompositions can be viewed as the transformed representation of a given tensor.  The number of parameters in CPD is $\mathcal{O}(dnr)$ that is linear to tensor order, however, its optimization problem is difficult and convergence is slow. The Tucker model is stable and can approximate an arbitrary tensor as close as possible, however, its number of parameters is  $\mathcal{O}(dnr+r^d)$ that is exponential to tensor order.  In contrast,  TT and TR decompositions have similar representation power to Tucker model, while their  number of paramters is $\mathcal{O}(dnr^2)$ that is linear to tensor order.

It should be noted that (i) TR model has a more generalized and powerful representation ability than TT model, due to  relaxation of the  strict condition $r_1=r_{d+1}=1$ in TT.  In fact, TT decomposition can be viewed as a special case of TR model, as demonstrated in Sec.~\ref{sec:TT}. (ii) TR-ranks are usually smaller than TT-ranks because TR model can be represented as a linear combination of TT decompositions whose cores are partially shared. (iii) TR model is more flexible than TT, because TR-ranks can be equally distributed in the cores, but TT-ranks have a relatively fixed pattern, i.e., smaller in border cores and larger in middle cores. (iv) Another important advantage of TR model over TT model is the circular dimensional permutation invariance (see Theorem \ref{theorem:invariance}). In contrast, the sequential  multilinear products of cores in TT must follow a strict order such that the optimized TT cores highly depend on permutation of the original tensor.

\subsection{CP decomposition}
The cannonical polyadic decomposition (CPD) aims to represent a $d$th-order tensor $\tensor T$ by a sum of rank-one tensors, given by
\begin{equation}
\tensor T = \sum_{\alpha =1}^r \mat u^{(1)}_\alpha \circ \cdots \circ \mat u^{(d)}_\alpha,
\end{equation}
where each rank-one tensor is represented by an outer product of $d$ vectors. It can be also written in the element-wise form given by
\begin{equation}
\label{eq:CPD2}
T(i_1, \ldots, i_d) = \left\langle  \mat u^{(1)}_{i_1}, \ldots, \mat u^{(d)}_{i_d}  \right\rangle,
\end{equation}
where $\langle\cdot,\ldots,\cdot \rangle$ denotes an inner product of a set of vectors, i.e., $\mat u^{(k)}_{i_k}\in\mathbb{R}^{r}, k=1,\ldots,d$.

By defining $\mat V_k(i_k) = \text{diag}(\mat u^{(k)}_{i_k})$ which is a diagonal matrix for each fixed $i_k$ and $k$, where $k=1,\ldots,d$, $i_k=1,\ldots,n_k$, we can rewrite (\ref{eq:CPD2}) as
\begin{equation}
T(i_1,\ldots,i_d) = \text{Tr}(\mat V_1(i_1) \mat V_2(i_2)\cdots \mat V_d(i_d)).
\end{equation}
Hence, CPD can be viewed as a special case of TR decomposition $\tensor T =\Re(\tensor V_1,\ldots, \tensor V_d)$ where the cores $\tensor V_k, k=1,\ldots, d$ are of size $r\times n_k \times r$ and each lateral slice matrix $\mat V_k(i_k)$ is a diagonal matrix of size $r\times r$.

\subsection{Tucker decomposition}
The Tucker decomposition aims to represent a $d$th-order tensor $\tensor T$ by a multilinear product between a core tensor $\tensor G\in\mathbb{R}^{r_1\times \cdots\times r_d}$ and factor matrices $\mat U^{(k)}\in\mathbb{R}^{n_k\times r_k}, k=1,\ldots,d$, which is expressed by
\begin{equation}
\label{eq:tucker}
\tensor T = \tensor G \times_1 \mat U^{(1)}\times_2\cdots\times_d \mat U^{(d)} = [\![ \tensor G, \mat U^{(1)}, \ldots,\mat U^{(d)}]\!].
\end{equation}

By assuming the core tensor $\tensor G$ can be represented by a TR decomposition $\tensor G = \Re(\tensor V_1, \ldots, \tensor V_d)$, the Tucker decomposition (\ref{eq:tucker})  in the element-wise form can be rewritten as

\begin{equation}
\begin{split}
&T(i_1,\ldots, i_d) \\
 &\quad= \Re(\tensor V_1, \ldots, \tensor V_d)\times_1 \mat u^{(1)T}(i_1)\times_2\cdots\times_d \mat u^{(d)T}(i_d)\\
&\quad= \text{Tr}\left\{ \prod_{k=1}^d \left( \sum_{\alpha_k=1}^{r_k} \mat V_k(\alpha_k)u^{(k)}(i_k,\alpha_k)\right) \right\}\\
&\quad = \text{Tr}\left\{ \prod_{k=1}^d   \left( \tensor V_k \times_2 \mat u^{(k)T}(i_k) \right)  \right\},
\end{split}
\end{equation}
where the second step is derived by applying Theorem \ref{theorem:TRtimesvectors}.
Hence, Tucker model can be  represented as a TR decomposition $\tensor T = \Re(\tensor Z_1, \ldots, \tensor Z_d) $ where the cores  are computed by the multilinear products between TR cores representing $\tensor G$ and the factor matrices, respectively, which is
\begin{equation}
\tensor Z_k = \tensor V_k \times_2 \mat U^{(k)}, \quad k=1,\ldots, d.
\end{equation}

\begin{table*}[htbp]
\renewcommand{\arraystretch}{1.3}
\caption{The functional data $f_1(x),f_2(x),f_3(x)$ is tensorized to 10th-order tensor ($4\times 4\times\ldots\times 4$). In the table, $\epsilon$, $\bar{r}$, $N_p$ denote relative error, average rank, and the total number of parameters, respectively.  }
\label{Tab:Simulation1}
\centering
\resizebox{\textwidth}{!}{
\begin{tabular}{cccccc c cccc c cccc c cccc}
\hline
 \multirow{2}{*}{}  &  \multicolumn{4}{c}{$f_1(x)$} &&  \multicolumn{4}{c}{$f_2(x)$ }  &&  \multicolumn{4}{c}{$f_3(x)$ }  &&  \multicolumn{4}{c}{$f_1(x)+\mathcal{N}(0,\sigma), SNR=60dB$ }\\
\cline{2-5} \cline{7-10} \cline{12-15} \cline{17-20}
& $\epsilon$ & $\bar{r}$ & $N_p$ &  Time (s)  & & $\epsilon$ & $\bar{r}$ & $N_p$ &  Time (s) && $\epsilon$ & $\bar{r}$ & $N_p$ &  Time (s) && $\epsilon$ & $\bar{r}$ & $N_p$ &  Time (s)\\
\hline
 TT-SVD & 3e-4 & 4.4 & 1032 & 0.17 &&  3e-4 & 5 & 1360 & 0.16  &&  3e-4 & 3.7 & 680 & 0.16  &&  1e-3 & 16.6 & 13064 & 0.5\\
 TR-SVD &  3e-4 & 4.4 & 1032 & 0.17  &&  3e-4 & 5 & 1360 & 0.28  &&  5e-4 & 3.6 & 668 & 0.15  &&  1e-3 & 9.7 & 4644 & 0.4\\
 TR-ALS  & 3e-4 & 4.4 & 1032 & 13.2   &&  3e-4 & 5 & 1360 & 18.6 &&  8e-4 & 3.6 & 668 & 4.0  &&  1e-3 & 4.4 & 1032 & 11.8 \\
  TR-ALSRA & 7e-4 & 5 & 1312 & 85.15 &&  2e-3 & 4.9 & 1448 & 150.5  &&  4e-3 & 4.6 & 1296 & 170.0  &&  1e-3 & 4.5 & 1100 & 64.8\\
  TR-BALS & 9e-4 & 4.3 & 1052 & 4.6 &&  8e-4 & 4.9 & 1324 & 5.7 &&  5e-4 & 3.7 & 728 & 3.4  &&  1e-3 & 4.2 & 1000 & 6.1\\
\hline
\end{tabular}
}
\end{table*}

\subsection{TT decomposition}
\label{sec:TT}
The tensor train decomposition aims to represent a $d$th-order tensor $\tensor T$ by a sequence of cores $\tensor G_k, k=1,\ldots, d$, where the first core $\mat G_1\in\mathbb{R}^{n_1\times r_2}$ and the last core $\mat G_d\in\mathbb{R}^{r_{d}\times n_d}$ are matrices while the other cores $\tensor G_k \in\mathbb{R}^{r_{k}\times n_k\times r_{k+1}}, k=2,\ldots,d-1$ are 3rd-order tensors. Specifically, TT decomposition in the element-wise form is expressed as
\begin{equation}
T(i_1,\ldots,i_d) = \mat g_1(i_1)^T \mat G_2(i_2)\cdots \mat G_{d-1}(i_{d-1}) \mat g_d(i_d),
\end{equation}
where $\mat g_1(i_1)$ is the $i_1$th row vector of $\mat G_1$, $\mat g_d(i_d)$ is the $i_d$th column vector of $\mat G_d$, and $\mat G_k(i_k), k=2,\ldots, d-1$ are the $i_k$th lateral slice matrices of $\tensor G_k$.

According to the definition of TR decomposition in (\ref{eq:TRD1}), it is obvious that TT decomposition is a special case of TR decomposition where the first and the last cores are matrices, i.e., $r_1=r_{d+1}=1$. On the other hand, TR decomposition can be also rewritten as
\begin{equation}
\begin{split}
&T(i_1,\ldots, i_d) =\text{Tr}\left\{\mat Z_1(i_1)\mat Z_2(i_2)\cdots \mat Z_d(i_d)\right\} \\
&= \sum_{\alpha_1 =1}^{r_1} \mat z_1(\alpha_1, i_1,:)^T \mat Z_2(i_2)\cdots\mat Z_{d-1}(i_{d-1})\mat z_d(:,i_d,\alpha_1)
\end{split}
\end{equation}
where $\mat z_1(\alpha_1, i_1,:) \in\mathbb{R}^{r_2}$ is the  $\alpha_1$th row vector of the matrix $\mat Z_1(i_1)$ and $\mat z_d(:,i_d,\alpha_1)$ is the $\alpha_1$th column vector of the matrix $\mat Z_d(i_d)$. Therefore, TR decomposition can be interpreted as a sum of TT representations. The number of TT representations is $r_1$ and these TT representations have the common cores $\tensor Z_k, \text{ for } k=2,\ldots, d-1$. In general, TR outperforms TT in terms of  representation power due to the fact of linear combinations of a group of TT representations. Furthermore, given a specific  approximation  level, TR representation requires smaller ranks than TT representation.

\section{Experimental Results}
\label{sec:experiment}
In the following section we investigate the performance of the proposed TR model and algorithms, while also comparing it with the existing tensor decomposition models.  We firstly conducted several numerical experiments on synthetic data to evaluate the effectiveness of our algorithms. Subsequently, two real-world datasets were employed to invesigate the representation power of TR model together with classification performances.
All computations were performed by a windows workstation with 3.33GHz Intel(R) Xeon(R) CPU, 64G memory and MATLAB R2012b.

\subsection{Numerical experiments}
\begin{figure}
  \centering
  \includegraphics[width=1\columnwidth]{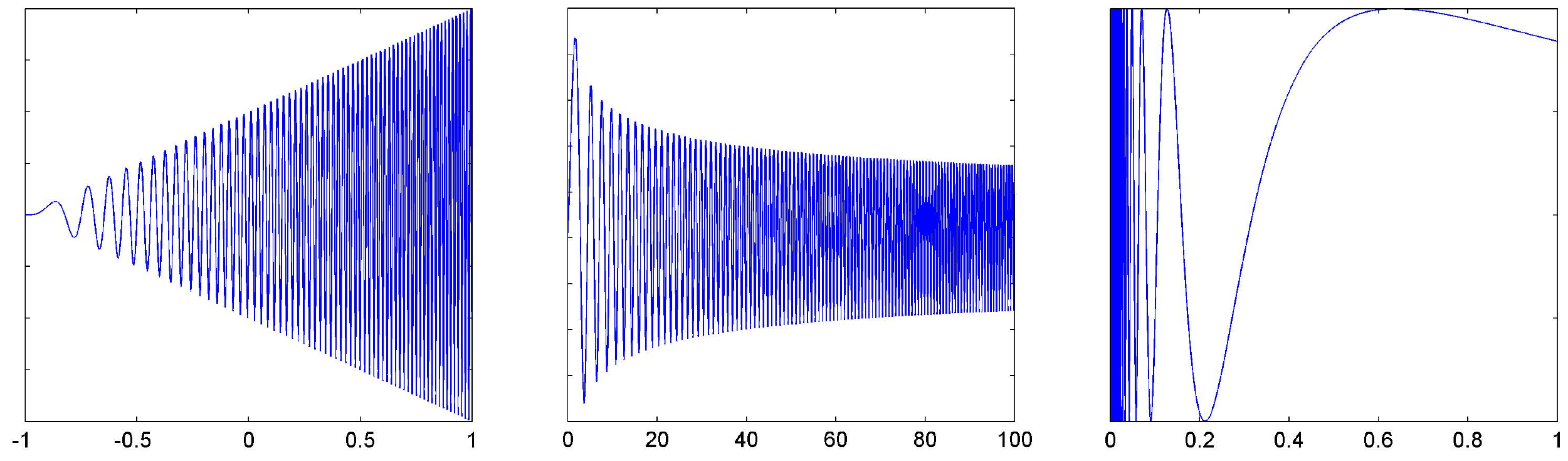}\\
  \caption{Highly oscillated functions. The left panel is $f_1(x)=(x+1)sin(100(x+1)^2)$. The middle panel is  Airy function: $f_2(x)=x^{-\frac{1}{4}} sin(\frac{2}{3}x^{\frac{3}{2}})$. The right panel is Chirp function $f_3(x)=sin\frac{x}{4}cos(x^2)$.  }
  \label{fig:functions}
\end{figure}

In the first example, we consider highly oscillating functions that can be approximated well by a low-rank TT format~\cite{khoromskij2015tensor,khoromskij2015fast}, as shown in Fig.~\ref{fig:functions}. We firstly apply the reshaping operation to the functional vector resulting in a $d$th-order tensor of size $n_1\times n_2\times\cdots\times n_d $, where isometric size is usually preferred, i.e., $n_1=n_2=\cdots=n_d = n$, with the total number of elements denoted by $N=n^d$. This operation is also termed as \emph{tensorization} or \emph{n-folding representation} of the functional vector. We tested the TR model on these functional data by using four proposed algorithms (i.e., TR-SVD, TR-ALS, TR-ALSAR, TR-BALS) and compared with TT-SVD algorithm in terms of relative error $\epsilon$, average rank $\bar{r}$, total number of parameters $N_p$ and running time.  In addition to three functional data, we also performed simulations on $f_1(x)$ when Gaussian noise was considered. We specify the error bound (tolerance), denoted by $\epsilon_p=10^{-3}$, as the stopping criterion for all compared algorithms.  These algorithms  can determine either TT-ranks or TR-ranks according to relative error tolerance $\epsilon_p$ except TR-ALS that requires manually given TR-ranks. We use TR-ranks obtained from TR-SVD algorithm as the setting of TR-ALS algorithm. Thus, we can compare the difference between TR approximations with orthogonal cores and non-orthogonal cores. As shown in Table~\ref{Tab:Simulation1}, TR-SVD and TT-SVD obtain comparable results including relative error, average rank and number of parameters for approximation of $f_1(x),f_2(x),f_3(x)$, while TR-SVD outperforms TT-SVD on $f_1(x)$ with Gaussian noise, which requires much smaller average TR-ranks. This result indicates that TR-SVD can represent such data by smaller number of parameters $N_p$ and effectively prevent overfitting the noise. The running time of TR-SVD and TT-SVD are comparable in all simulations. Although TR-ALS does not impose orthogonality constraints on cores, it can also achieve the same relative error with TR-SVD for approximation of $f_1(x),f_2(x),f_3(x)$, while its running time is longer than TR-SVD due to the iterative procedure. When Gaussian noise was considered, TR-ALS using fixed TR-ranks can effectively avoid overfitting to noise. Both TR-ALSAR and TR-BALS can adjust TR-ranks automatically based on error bound $\epsilon_p$, however, TR-BALS outperforms TR-ALSAR in terms of relative error, average rank and running time. TR-BALS can obtain comparable results with TR-SVD for approximation of $f_1(x),f_2(x),f_3(x)$ and significantly outperforms TR-SVD when noise is involved. More detailed comparisons of these algorithms can be found in Table~\ref{Tab:Simulation1}.

\begin{table*}[htbp]
\renewcommand{\arraystretch}{1.3}
\caption{The results under different shifts of dimensions on functional data $f_2(x)$ with error bound setting to $10^{-3}$ . For the 10th-order tensor, all 9 dimension shifts were considered and the average rank $\bar{r}$ as well as the number of total parameters $N_p$ are compared. }
\label{tab:Simulation2}
\centering
\begin{tabular}{c ccccccccc c ccccccccc}
\hline
\multirow{2}{*}{} & \multicolumn{9}{c}{$\bar{r}$} && \multicolumn{9}{c}{$N_p$} \\
 \cline{2-10} \cline{12-20}
 & 1 & 2 & 3 &  4 & 5 & 6 & 7 & 8 & 9 &&  1 & 2 & 3 &  4 & 5 & 6 & 7 & 8 & 9\\
 \hline
 TT-SVD & 5.2 & 5.8 & 6 &  6.2 & 7 & 7 & 8.5 & 14.6 & 8.4 &&  1512 & 1944 & 2084 &  2144 & 2732 & 2328 & 3088 & 10376 & 3312\\
  TR-SVD & 5.2 & 5.8 & 5.9 &  6.2 & 9.6 & 10 & 14 & 12.7 & 6.5 &&  1512 & 1944 & 2064 &  2144 & 4804 & 4224 & 9424 & 7728 & 2080\\
   TR-ALS & 5 & 5 & 5 &  5 & 5 & 5 & 5 & 5 & 5 &&  1360 & 1360 & 1360 &  1360 & 1360 & 1360 & 1360 & 1360 & 1360\\
    TR-ALSAR & 5.5 & 6.6 & 6.2 &  5.2 & 5.3 & 5.8 & 6.9 & 5.3 & 4.7 &&  1828 & 2064 & 1788 &  1544 & 1556 & 1864 & 2832 & 1600 & 1324\\
     TR-BALS & 5 & 4.9 & 5 &  4.9 & 4.9 & 5 & 5 & 4.8 & 4.9 &&  1384 & 1324 & 1384 &  1348 & 1348 & 1384 & 1384 & 1272 & 1324\\
\hline
\end{tabular}
\end{table*}

It should be noted that TT representation has the property that $r_1=r_{d+1}=1$ and $r_k, k=2,\ldots,d-1$ are bounded by the rank of $k$-unfolding matrix of $\mat T_{\langle k\rangle}$, which limits its generalization ability and consistence when the dimensions of tensor have been shifted or permutated.  To demonstrate this, we consider to shift the dimensions of $\tensor T$ of size ${n_1\times \cdots\times n_d}$ by $k$ times leading to $\overleftarrow{\tensor{T}}^k$ of size $n_{k+1}\times \cdots\times n_d\times n_1\times\cdots\times n_{k}$. We applied TT and TR decompositions to the dimension shifted data generated by $f_2(x)$, and compared different algorithms under all situations when $k=1,\ldots,9$. As shown in Table~\ref{tab:Simulation2}, the average TT-ranks are varied dramatically along with the different shifts. In particular, when $k=8$, $\bar{r}_{tt}$ becomes 14.6, resulting in the large number of parameters $N_p=10376$, which implies a very low compression ability. In contrast to TT decomposition, TR-SVD achieves slightly better performance in some cases but not for all cases. For TR-ALS, since TR-ranks were specified as $(r_{k+1},\ldots,r_d,r_1,\ldots,r_k)$ for any $k$ shifted cases, it achieves the consistent results. However, the TR-ranks are usually unknown in practice, we must resort to TR-ALSAR and TR-BALS that can adapt TR-ranks automatically based on the error tolerance. As compared with TR-ALSAR, TR-BALS can obtain more compact representation together with smaller relative errors. In addition, TR-BALS can even outperform TR-ALS in several cases, implying that TR-ranks obtained from TR-SVD are not always the optimal one. More detailed results can be found in Table~\ref{tab:Simulation2}. These experiments demonstrate that TR decomposition is stable, flexible and effective for general data, while TT decomposition has strict limitations on data organization. Therefore, we can conclude that TR model is a more generalized and powerful representation with higher compression ability as compared to TT.

\begin{table*}[htbp]
\renewcommand{\arraystretch}{1.3}
\caption{The detailed results on synthetic data of size ($n_1\times n_2\times\ldots n_d$) where $n_1=\cdots=n_d=4, d=10, \mathbf{r}_{true}=\{1,\ldots,4\}$. In the table, $\epsilon$ denotes relative error, $r_{max}$ is the maximum rank, and $N_p$ denotes the number of total parameters (i.e., model complexity).  'N/A' denotes that rank is specified manually as the true rank.  }
\label{Tab:Simulation3}
\centering
\begin{tabular}{ccccccc c cccc}
\hline
\multirow{2}{*}{$\mathbf{r}_{true}$}  & \multirow{2}{*}{Methods}  &  \multicolumn{4}{c}{Without Noise} &&  \multicolumn{4}{c}{Gaussian Noise (SNR=40dB)} \\
\cline{3-6} \cline{8-11}
  &   & $\epsilon$ & $r_{max}$ & $N_p$ &  Time (s)  & & $\epsilon$ & $r_{max}$ & $N_p$  &  Time (s)  \\
\hline
\multirow{5}{*}{1}
& TT-SVD & 8e-15 & 1 & 40 & 0.12 &&  8e-3 & 361 & 5e5 & 0.9\\
 & TR-SVD &  7e-15 & 1 & 40 & 0.14  &&  8e-3 & 323 & 6e5 & 1.2\\
 & TR-ALS  & 3e-14 & N/A & 40 & 1.10   &&  1e-2 & N/A & 40 & 1.0 \\
 & TR-ALSAR & 3e-14 & 1 & 40 & 2.29 &&  1e-2 & 1 & 40 & 2.2\\
 & TR-BALS & 6e-15 & 1 & 40 & 0.95 &&  1e-2 & 1 & 40 & 1.0\\
\hline
\multirow{5}{*}{2}
  & TT-SVD & 2e-14 & 4 & 544 & 0.24 &&  8e-3 & 362 & 5e5 & 0.9\\
 & TR-SVD &  2e-14 & 8 & 1936 & 0.49 &&  7e-3 & 324 & 6e5 & 1.2\\
 & TR-ALS  & 2e-7 & N/A & 160 & 3.1 &&  1e-2 & N/A & 160 & 2.5\\
 & TR-ALSAR & 2e-5 & 4 & 284 & 5.70 &&  1e-2 & 4 & 320 & 6.1\\
 & TR-BALS & 2e-5 & 2 & 160 & 1.17 &&  1e-2 & 2 & 160 & 1.2\\
\hline
\multirow{5}{*}{3}
 & TT-SVD & 2e-14 & 9 & 2264 & 0.30 &&  8e-3 & 364 & 5e5 & 1.0\\
 & TR-SVD &  2e-14 & 18 & 7776 & 0.55 &&  7e-3 & 325 & 6e5 & 1.2\\
 & TR-ALS  & 9e-6 & N/A & 360 & 5.39 &&  1e-2 & N/A & 360 & 4.7\\
 & TR-ALSAR & 8e-4 & 4 & 580 & 10.66 &&  1e-2 & 4 & 608 & 11.8\\
 & TR-BALS & 9e-4 & 3 & 360 & 2.24 &&  1e-2 & 3 & 360 & 2.4\\
\hline
\multirow{5}{*}{4}
  & TT-SVD & 2e-14 & 16 & 6688 & 0.45 &&  7e-3 & 395 & 6e5 & 1.2\\
 & TR-SVD &  4e-4 & 32 & 22352 & 0.66 &&  7e-3 & 335 & 7e5 & 1.2\\
 & TR-ALS  & 1e-6 & N/A & 640 & 14.00 &&  1e-2 & N/A & 640 & 11.4\\
 & TR-ALSAR & 8e-4 & 7 & 1084 & 30.99 && 1e-2 & 7 & 1252 & 35.9\\
 & TR-BALS & 2e-4 & 4 & 640 & 3.90 &&  1e-2 & 4 & 640 & 3.9\\
\hline
\end{tabular}
\end{table*}

In the next experiment, we consider  higher order tensors which are known to be represented well by TR model. We simplify the TR-ranks as $r_1=r_2=\cdots=r_d$ that are varied from 1 to 4, $n_1=n_2=\cdots=n_d=4$ and $d=10$. The cores, $\tensor G_k, (k=1,\ldots,d)$, were drawn from the normal distribution, which are thus used to generate  a 10th-order tensor. We firstly apply different algorithms with the setting of $\epsilon_p=10^{-3}$ to $\tensor T$ generated by using different ranks. Subsequently, we also consider Gaussian noise corrupted tensor $\tensor T+\mathcal{N}(0,\sigma^2)$ with SNR=40dB and apply these algorithms with the setting of $\epsilon_p=10^{-2}$. As shown in Table~\ref{Tab:Simulation3}, the maximum rank of TT-SVD  increases dramatically when the true rank becomes larger and is approximately $r^2_{true}$, which thus results in a large number of parameters $N_p$ (i.e., low compression ability). TR-SVD performs similarly to TT-SVD, which also shows low compression ability when the true rank is high. For TR-ALS, since the true rank is given manually, it shows the best result and can be used as the baseline to evaluate the other TR algorithms. In contrast to TT-SVD and TR-SVD, both TR-ALSAR and TR-BALS are able to adapt TR-ranks according to $\epsilon_p$, resulting in the significantly lower rank reflected by $r_{max}$ and lower model complexity reflected by $N_p$.  As compared to TR-BALS, TR-ALSAR is prone to overestimate the rank and computation cost is relatively high. The experimental results show that TR-BALS can learn the TR-ranks correctly in all cases, and the number of parameters $N_p$ are exactly equivalent to the baseline, meanwhile, the running time is also reasonable. For the noisy tensor data, we observe that TT-SVD and TR-SVD are affected significantly with $r_{max}$ becoming 361 and 323 when true rank is only 1, which thus results in a poor compression ability. This indicates that TT-SVD and TR-SVD are sensitive to noise and prone to overfitting problems. By contrast, TR-ALS, TR-ALSAR, and TR-BALS obtain impressive results that are similar to that in noise free cases. TR-ALSAR slightly overestimates the TR-ranks. It should be noted that TR-BALS can estimate the true rank correctly and obtain the best compression ratio as TR-ALS given true rank. In addition, TR-BALS is more computationally efficient than TR-ALSAR. In summary, TT-SVD and TR-SVD have limitations for representing the tensor data with symmetric ranks, and this problem becomes more severe when noise is considered. The ALS algorithm can avoid this problem due to the flexibility on distribution of ranks. More detailed results can be found in Table~\ref{Tab:Simulation3}.

\subsection{COIL-100 dataset}
\begin{figure}[h]
  \centering
  \includegraphics[width=1\columnwidth]{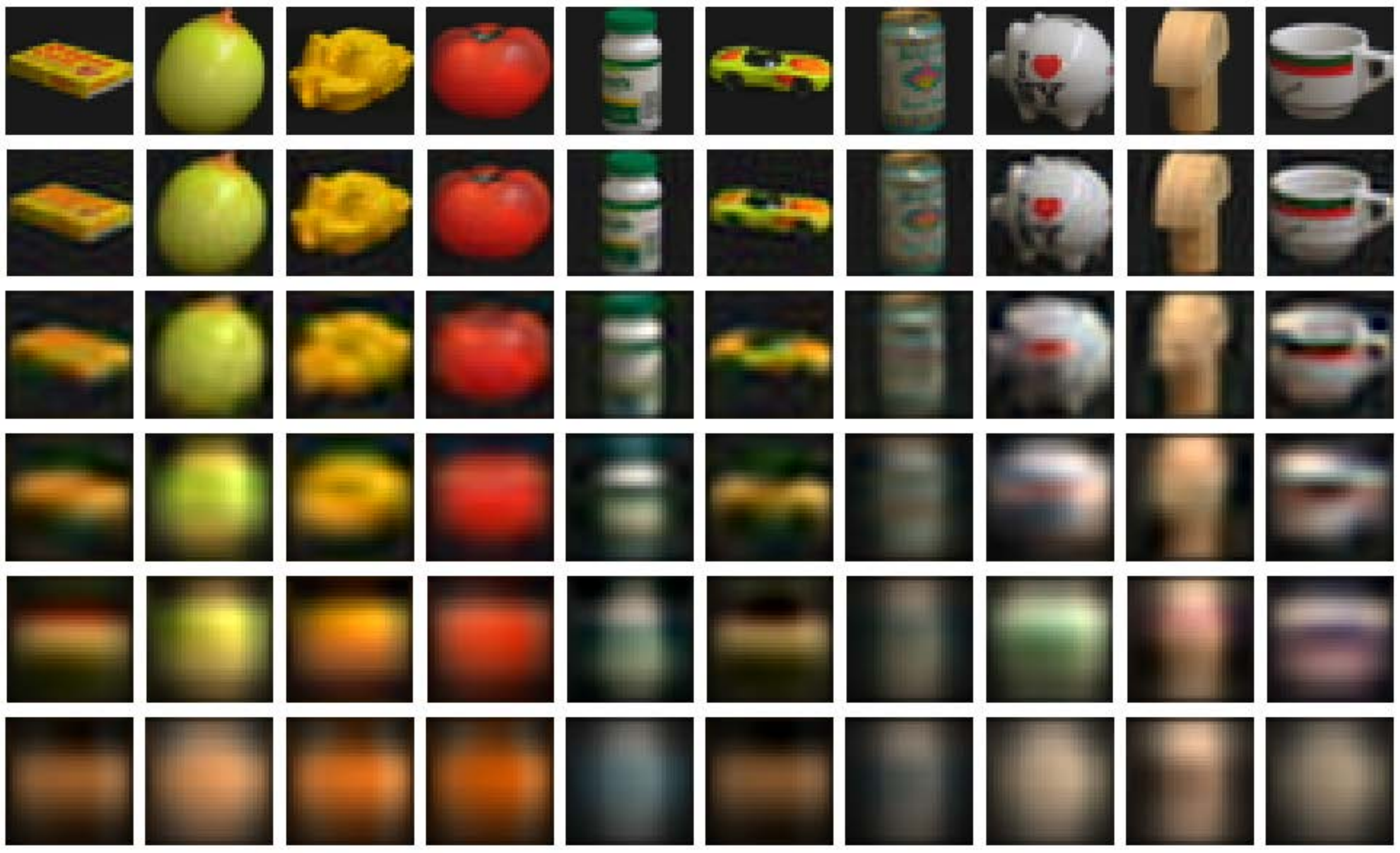}\\
  \caption{ The reconstruction of Coil-100 dataset by using TRSVD. The top row shows the original images, while the reconstructed images are shown from the second to sixth rows corresponding to $\epsilon$=0.1, 0.2, 0.3, 0.4, 0.5, respectively.}
  \label{fig:TRSVD-Coil100}
\end{figure}

\begin{table}[h]
\renewcommand{\arraystretch}{1.1}
\caption{The comparisons of different algorithms on Coil-100 dataset. $\epsilon$, $r_{max}$,$\bar{r}$ denote relative error, the maximum rank and the average rank, respectively. For classification task, the ratio of training samples $\rho=50 \%, 10\%$ were considered. }
\label{tab:coil100}
\centering
\begin{tabular}{c c c c c c c}
\hline
  & $\epsilon$ & $r_{max}$ & $\bar{r}$ & \pbox{1in} {Acc. (\%) \\($\rho=50 \%$)} & \pbox{1in} {Acc. (\%)\\ ($\rho=10 \%$)} \\
\hline
\multirow{4}{*}{ CP-ALS}
& 0.20 & 70 & 70 & 97.46  & 80.03 \\
& 0.30 & 17 & 17 & 97.56  & 83.38 \\
& 0.39 & 5 & 5 & 90.40  & 77.70 \\
& 0.47 & 2 & 2 & 45.05  & 39.10 \\
\hline
\multirow{5}{*} {TT-SVD}
 & 0.19 & 67  & 47.3 & 99.05  & 89.11  \\
 & 0.28 & 23  & 16.3 & 98.99  & 88.45  \\
 & 0.37 & 8   & 6.3  & 96.29  & 86.02  \\
 & 0.46 & 3   & 2.7  & 47.78  & 44.00  \\
 \hline
\multirow{5}{*} {TR-SVD}
 & 0.19 & 23  & 12.0 & 99.14  & 89.29  \\
 & 0.28 & 10  & 6.0  & 99.19 & 89.89  \\
 & 0.36 & 5   & 3.5  & 98.51  & 88.10  \\
 & 0.43 & 3   & 2.3  & 83.43  & 73.20 \\
 \hline
\multirow{4}{*} {TR-ALS}
 & 0.20 & 23 & 12.0 & 95.10  & 70.38  \\
 & 0.30 & 10 & 6.0 & 97.32  & 80.71  \\
 & 0.40 & 5 & 3.5 & 95.77  & 79.92  \\
 & 0.47 & 3 & 2.3 & 65.73  & 52.25  \\
 \hline
\multirow{3}{*} {TR-ALSRA}
 & 0.20 & 51 & 16.2 & 79.49  & 60.18\\
 & 0.30 & 11 & 6 & 93.11  & 79.48  \\
 & 0.39 & 4 & 3 & 83.67   & 66.50  \\
 & 0.47 & 2 & 2 & 76.02  & 62.79  \\
 \hline
\multirow{2}{*} {TR-BALS}
 & 0.19 & 31 & 11 & 96.00 & 72.75\\
 & 0.29 & 13 & 7 & 96.41   & 74.32 \\
 & 0.40 & 4 & 2   & 94.68  & 84.22  \\
 & 0.45 & 2 & 1.5 & 88.30  & 76.86  \\
\hline
\end{tabular}
\end{table}

In this section, the proposed TR algorithms are evaluated and compared with TT and CP decompositions on Columbia Object Image Libraries (COIL)-100 dataset~\cite{nayar1996columbia} that contains 7200 color images of 100 objects (72 images per object) with different reflectance and complex geometric characteristics. Each image can be represented by a 3rd-order tensor of size $128\times 128\times 3$ and then is downsampled to $32\times 32\times 3$. Hence, the dataset can be finally organized as a 4th-order tensor of size $32\times 32\times 3\times 7200$. In Fig.~\ref{fig:TRSVD-Coil100}, we show the reconstructed images under different relative errors $\epsilon\in\{0.1,\ldots,0.5\}$ which correspond to different set of TR-ranks $\mathbf{r}_{TR}$. Obviously, if $\mathbf{r}_{TR}$ are small, the images are smooth and blurred, while the images are more sharp when $\mathbf{r}_{TR}$ are larger. This motivates us to apply TR model to extract the abstract information of objects by using low-dimensional cores, which can be considered as the feature extraction approach. More specifically, the 4th TR core $\tensor G_4$ of size $(r_4 \times 7200 \times r_1)$, can be used as the latent TR features while the subchain $\tensor G^{\neq 4}$ can be considered as the basis of latent subspace. It should be emphasized that the feature extraction by TR decomposition has some essentially different characteristics. The number of features is determined by $r_4\times r_1$, while the flexibility of subspace basis is determined by $r_2, r_3$. This implies that we can obtain very different features even the number of features is fixed, by varying $r_2, r_3$ of TR-ranks that are related to the subspace basis. Therefore, TR decomposition might provide a flexible and effective feature extraction framework. In this experiment, we simply use relative error bound as the criterion to control TR ranks, however, TR ranks are also possible to be controlled individually according to the specific physical meaning of each dimension. We apply the proposed TR algorithms with the setting of $\epsilon_p\in\{0.2,\ldots,0.5\}$, then the core $\tensor G_4$ is used as extracted features with reduced dimensions. Subsequently, we apply the K-nearest neighbor (KNN) classifier with K=1 for classification. For detailed comparisons, we randomly select a certain ratio $\rho=50\%$ or $\rho=10\%$ samples as the training set and the rest as the test set. The classification performance is averaged over 10 times of randomly splitting. In Table \ref{tab:coil100}, we show the relative error, maximum rank, average rank and classification accuracy under two different settings. We can see that $r_{max}$ of TT-SVD is mostly smaller than that of CP-ALS when their $\epsilon$ is in the same level, while $r_{max}$ of TR decompositions are much smaller than TT-SVD and CP-ALS. This indicates that TR model can outperform TT and CP models in terms of compression ratio. The best classification performance of CP-ALS is 97.56\% ($\rho=50\%$) and 83.38\% ($\rho=10\%$), which corresponds to $\epsilon=0.3$. TT-SVD achieves the classification performance of 99.05\% ($\rho=50\%$) and $89.11\%$ ($\rho=10\%$) when $\epsilon\approx 0.2$. TR-SVD can achieve 99.19\% ($\rho=50\%$) and $89.89\%$ ($\rho=10\%$) when $\epsilon \approx 0.3$. It should be noted that,  TR model, as compared to TT, can preserve the discriminant information well even when the fitting error is larger.  For TR-ALS, TR-ALSAR, and TR-BALS, $r_{max}$ is comparable to that of TT-SVD under the corresponding $\epsilon$, while the classification performances are slightly worse than TR-SVD. More detailed results are compared in Table \ref{tab:coil100}. This experiment demonstrates that TR and TT decompositions are effective for feature extractions and outperform the CP decomposition. In addition, TR decompositions achieve the best classification performances together with the best compression ability as compared to TT and CP decompositions.

\subsection{KTH video dataset}

\begin{figure}[htbp]
  \centering
  \includegraphics[width=1\columnwidth]{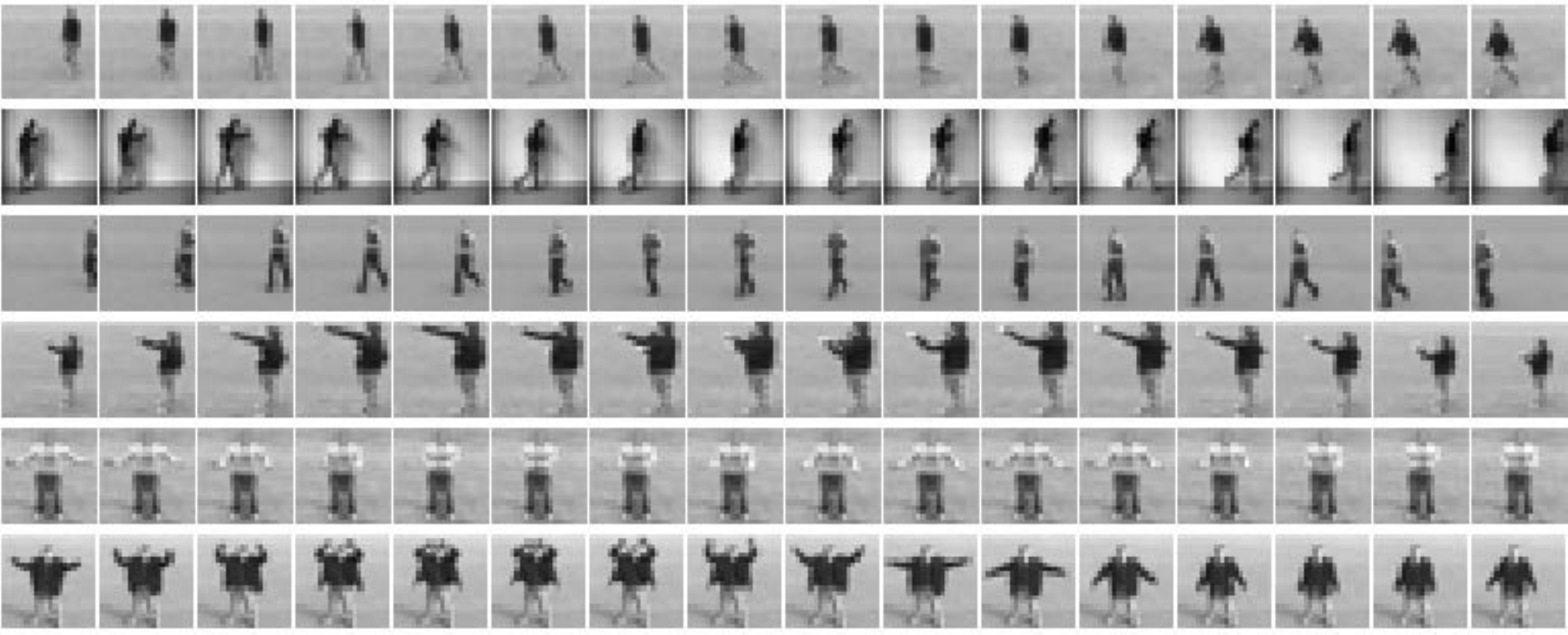}\\
  \caption{Video dataset consists of six types of human actions performed by 25 subjects in four different scenarios. From the top to bottom, six video examples corresponding to each type of actions are shown.  }
  \label{fig:KTHdataset}
\end{figure}

\begin{figure}[htbp]
  \centering
  \includegraphics[width=1\columnwidth]{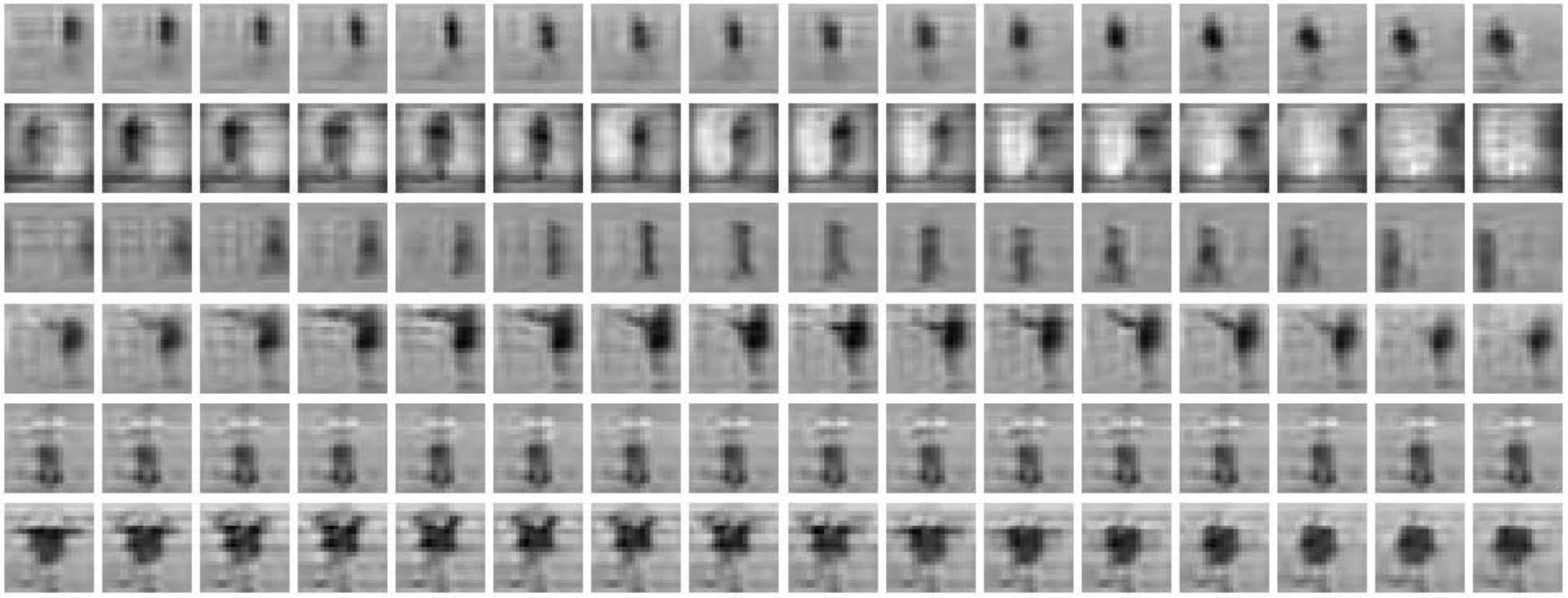}\\
  \caption{The six examples of reconstructed video sequences by TR-BALS with $\epsilon=0.27$, $r_{max}=24$, $\bar{r}=13.3$. The classification accuracy is 87.0\% by using the cores obtained from TR-BALS.   }
  \label{fig:KTHdataset2}
\end{figure}

\begin{table}[h]
\renewcommand{\arraystretch}{1.1}
\caption{The comparisons of different algorithms on KTH dataset. $\epsilon$ denotes the obtained relative error; $r_{max}$ denotes maximum rank; $\bar{r}$ denotes the average rank; $N_f$ denotes the total number of extracted features, and Acc. is the classification accuracy.  }
\label{tab:KTH}
\centering
\begin{tabular}{p{15mm} c c c c c c}
\hline
 & $\epsilon$ & $r_{max}$ & $\bar{r}$ & $N_f$ & Acc. ($5\times 5$-fold) \\
\hline
\multirow{3}{*}{ CP-ALS}
& 0.20 & 300 & 300 & 300 & 80.8 \% \\
& 0.30 & 40 & 40 & 40 & 79.3 \%\\
& 0.40 & 10 & 10 & 10 & 66.8 \%\\
\hline
\multirow{4}{*} {TT-SVD}
 & 0.20 & 139  & 78.0 & 139  & 84.8 \% \\
 & 0.29 & 38  & 27.3 & 36  & 83.5 \% \\
 & 0.38 & 14   & 9.3  & 9  & 67.8 \% \\
 \hline
\multirow{3}{*} {TR-SVD}
 & 0.20 & 99  & 34.2 & 297  & 78.8 \% \\
 & 0.29 & 27  & 12.0  & 81  & 87.7 \% \\
 & 0.37 & 10   & 5.8  & 18 & 72.4 \% \\
 \hline
\multirow{2}{*} {TR-ALS}
 & 0.30 & 27 & 12.0 & 81 & 87.3 \% \\
  & 0.39 & 10   & 5.8  & 18 & 74.1 \% \\
 \hline
\multirow{2}{*} {TR-ALSRA}
 & 0.29 & 29 & 16.0 & 39 & 82.3 \% \\
 & 0.39 & 8 & 5.3 & 16 & 74.1 \% \\
 \hline
\multirow{2}{*} {TR-BALS}
 & 0.27 & 24 & 13.3   & 100 & 87.0 \% \\
 & 0.40 & 11 & 6 & 44 & 82.9 \% \\
\hline
\end{tabular}
\end{table}

In this section, we test the TR decompositions on KTH video database~\cite{laptev2006local} containing six types of human actions (walking, jogging, running, boxing, hand waving and hand clapping) performed several times by 25 subjects in four different scenarios: outdoors, outdoors with scale variation, outdoors with different clothes and indoors as illustrated in Fig.~\ref{fig:KTHdataset}. There are 600 video sequences for each combination of 25 subjects, 6 actions and 4 scenarios. Each video sequence was downsampled to $20\times 20\times 32$. Finally, we can organize the dataset as a tensor of size $20\times 20\times 32\times 600$. We apply TR decompositions to represent the whole dataset by a set of 3rd-order TR cores, which can be also considered as the feature extraction or dimension reduction approach, and compare with TT and CP decompositions in terms of compression ability and classification performance. For extensive comparisons, we choose different error bound $\epsilon_p\in \{0.2,0.3,0.4\}$ for tensor decompositions.  In Table \ref{tab:KTH}, we can see that TR representations achieve better compression ratio reflected by smaller $r_{max}, \bar{r}$ than that of TT-SVD, while TT-SVD achieves better compression ratio than CP-ALS. For instance, when $\epsilon\approx 0.2$, CP-ALS requires $r_{max}=300$, $\bar{r}=300$; TT-SVD requires $r_{max}=139$, $\bar{r}=78$, while TR-SVD only requires $r_{max}=99$, $\bar{r}=34.2$. For comparisons of different TR algorithms, we observe that TR-BALS outperforms the other algorithms in terms of compression ability. However, TR-ALSRA and TR-BALS cannot approximate data with any given error bound. For classification performance, we observe that the best accuracy ($5\times 5$-fold cross validation) achieved by CP-ALS, TT-SVD, TR-SVD, TR-ALS, TR-ALSAR, TR-BALS are 80.8\%, 84.8\%, 87.7\%, 87.3\%, 82.3\%, 87.0\%, respectively. Note that these classification performances might not  be the state-of-the-art on this dataset, however, we mainly focus on the comparisons among CP, TT, and TR decomposition frameworks. To obtain the best performance, we may apply the specific supervised feature extraction methods to TT or TR representations of dataset. It should be noted that TR decompositions achieve the best classification accuracy when $\epsilon= 0.3$, while TT-SVD and CP-ALS achieve their best classification accuracy when $\epsilon=0.2$. This indicates that TR decomposition can preserve more discriminant information even when the approximation error is relatively high. Fig.~\ref{fig:KTHdataset2} illustrates the reconstructed video sequences by TR-BALS, which corresponds to its best classification accuracy. Observe that although the videos are blurred and smooth, the discriminative information for action classification is still preserved. The detailed results can be found in Table \ref{tab:KTH}.   This experiment demonstrates that TR decompositions are effective for unsupervised feature representation due to their flexibility of TR-ranks and high compression ability.

\section{Conclusion}
\label{sec:conclusion}
We have proposed a tensor decomposition model, which provides an efficient representation for a large-dimensional tensor by a sequence of low-dimensional cores. The number of parameters is $\mathcal{O}(dnr^2)$ that scales linearly to the tensor order. To optimize the latent cores, we have presented four different algorithms. In particular, TR-SVD is a non-recursive algorithm that is stable and efficient. TR-ALS is precise but requires TR-ranks to be given mannually. TR-ALSAR and TR-BALS can adapt TR-ranks automatically with relatively high computational cost. Furthermore, we have investigated  the properties on how the basic multilinear algebra can be performed efficiently by direct operations over TR representations (i.e., cores), which provides a potentially powerful framework for processing large-scale data. The relations to other tensor decomposition models are also investigated, which allows us to conveniently transform the latent representations from the traditional models to TR model. The experimental results have verified the effectiveness of the proposed TR model and algorithms.


\bibliographystyle{IEEEtran}
\bibliography{IEEEabrv,TensorNetwork}

\begin{IEEEbiography}[{\includegraphics[width=1in,height=1.25in,clip,keepaspectratio]{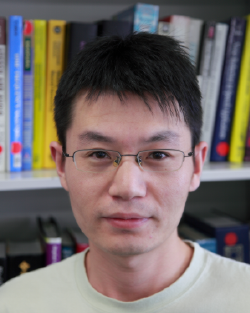}}]{Qibin Zhao} received the Ph.D. degree from Department of Computer Science and Engineering, Shanghai Jiao Tong University, Shanghai, China, in 2009. He is currently a research scientist at Laboratory for Advanced Brain Signal Processing in RIKEN Brain Science Institute, Japan and is also a visiting professor in Saitama Institute of Technology, Japan. His research interests include machine learning, tensor factorization, computer vision and brain computer interface. He has published more than 50 papers in international journals and conferences.
\end{IEEEbiography}
\vspace{-0in}

\begin{IEEEbiography}[{\includegraphics[width=1in,height=1.25in,clip,keepaspectratio]{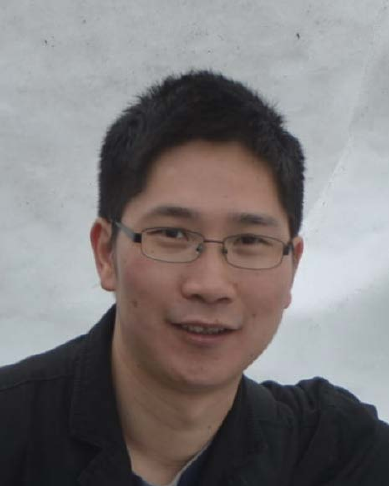}}]{Guoxu Zhou} received the Ph.D degree in intelligent signal and information processing from South China University of Technology, Guangzhou, China, in 2010. He is currently a research scientist of the laboratory for Advanced Brain Signal Processing, at RIKEN Brain Science Institute (JAPAN). His research interests include statistical signal processing, tensor analysis, intelligent information processing, and machine learning.
\end{IEEEbiography}

\vspace{-0in}
\begin{IEEEbiography}[{\includegraphics[width=1in,height=1.25in,clip,keepaspectratio]{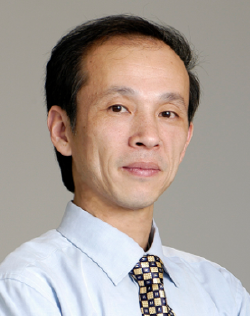}}]{Liqing Zhang} received the Ph.D. degree from Zhongshan University, Guangzhou, China, in 1988. He is now a Professor with Department of Computer Science and Engineering, Shanghai Jiao Tong University, Shanghai, China. His current research interests cover computational theory for cortical networks, visual perception and computational cognition, statistical learning and inference. He has published more than 210 papers in international journals and conferences.
\end{IEEEbiography}

\vspace{-0in}
\begin{IEEEbiography}[{\includegraphics[width=1in,height=1.25in,clip,keepaspectratio]{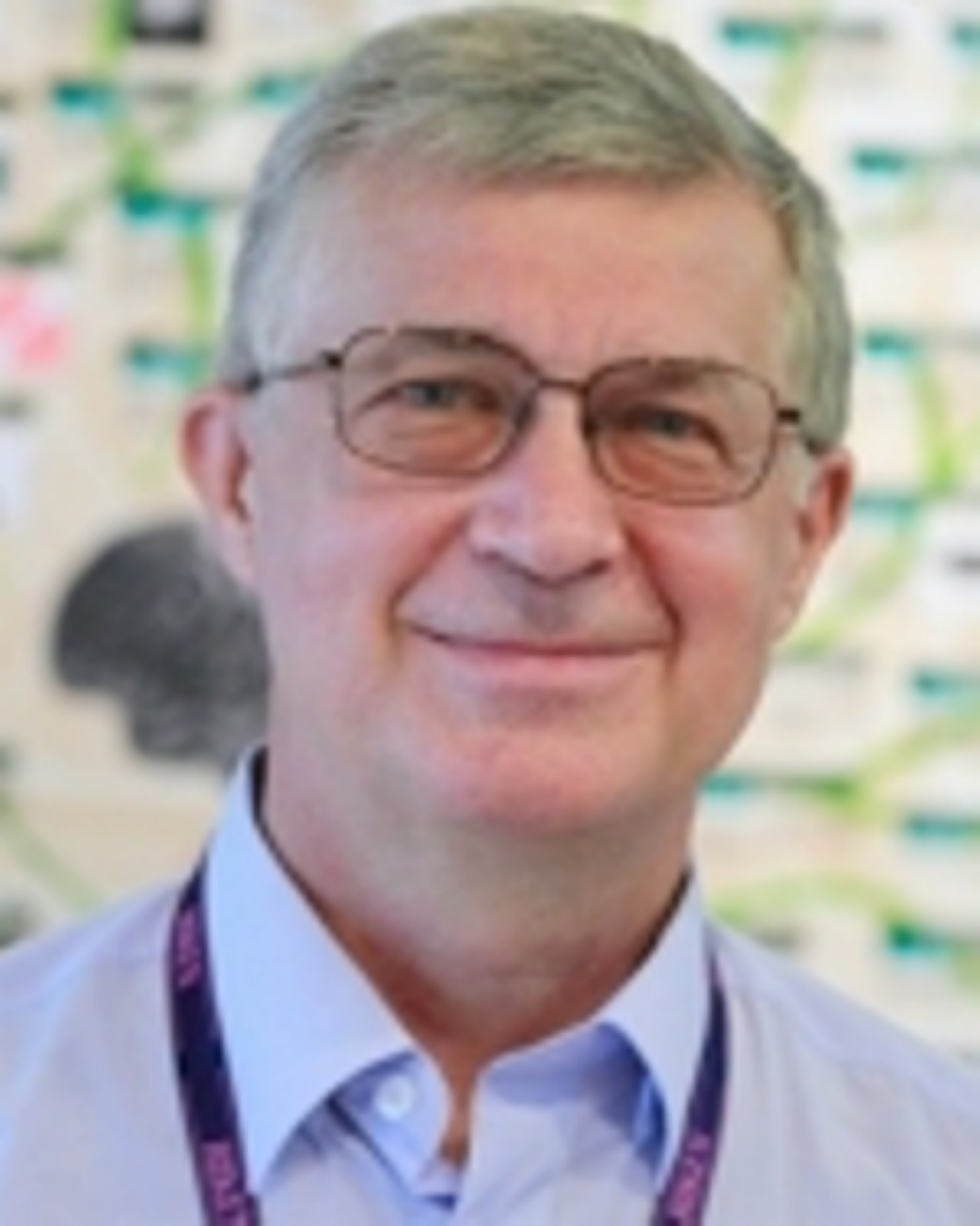}}]{Andrzej Cichocki} received the Ph.D. and Dr.Sc. (Habilitation) degrees, all in electrical engineering, from Warsaw University of Technology (Poland). He is the senior team leader of the Laboratory for Advanced Brain Signal Processing, at RIKEN BSI (Japan). He is coauthor of more than 400 scientific papers and 4 monographs (two of them translated to Chinese). He served as AE of IEEE Trans. on Signal Processing, TNNLS, Cybernetics and J. of Neuroscience Methods.
\end{IEEEbiography}
\vfill


\end{document}

%% file: DefinedComand.tex
\newcommand{\tensor}[1]{\boldsymbol{\mathcal{#1}}}

\newcommand{\mat}[1]{\mathbf{#1}}
\newcommand{\vect}[1]{\mathbf{#1}}

\theoremstyle{plain}
\newtheorem{theorem}{Theorem}[section]

\theoremstyle{definition}
\newtheorem{definition}[theorem]{Definition}
\theoremstyle{remark}